\documentclass[11pt]{article}
\usepackage{mathrsfs}
\usepackage{amsmath}
\usepackage{amsfonts}
\usepackage{enumerate}
\usepackage{setspace}
\usepackage{color}
\usepackage{bbm}
\usepackage{amsthm}
\usepackage{amssymb}
\usepackage{makecell}
\usepackage{graphicx}
\usepackage{epstopdf}

\newtheorem{theorem}{Theorem}[section]
\newtheorem{proposition}[theorem]{Proposition}
\newtheorem{definition}[theorem]{Definition}
\newtheorem{lemma}[theorem]{Lemma}
\newtheorem{remark}[theorem]{Remark}

\begin{document}
\title{\bf{Indefinite Linear-Quadratic Partially Observed Mean-Field Game\footnote{This research was supported by National Key R\&D Program of China [Grant 2023YFA1009200], the National Natural Science Foundation of China [Grants 12401583, 12022108 and 11971267], the Shandong Provincial Natural Science Foundation [Grants ZR2022JQ01 and ZR2024MA095], the Basic Research Program of Jiangsu [Grant BK20240416], the Taishan Scholars Climbing Program of Shandong [Grant TSPD20210302], and the Distinguished Young Scholars Program of Shandong University.}}}
\author{Tian Chen\footnote{School of Mathematics, Shandong University, Jinan, China,  chentian43@sdu.edu.cn, nietianyang@sdu.edu.cn, wuzhen@sdu.edu.cn} \quad Tianyang Nie\footnotemark[2] \quad Zhen Wu\footnotemark[2]}

\maketitle

\vspace{2mm}\noindent\textbf{Abstract.}\quad This paper investigates an indefinite linear-quadratic partially observed mean-field game with common noise, incorporating both state-average and control-average effects. In our model, each agent's state is observed through both individual and public observations, which are modeled as general stochastic processes rather than Brownian motions. {It is noteworthy that} the weighting matrices in the cost functional are allowed to be indefinite. We derive the optimal decentralized strategies using the Hamiltonian approach and establish the well-posedness of the resulting Hamiltonian system by employing a relaxed compensator. The associated consistency condition and the feedback representation of decentralized strategies are also established. Furthermore, we demonstrate that the set of decentralized strategies form an $\varepsilon$-Nash equilibrium. As an application, we solve a mean-variance portfolio selection problem.

\vspace{2mm} \noindent\textbf{Keywords.}\quad Mean-field game, Partially observed, Indefinite linear-quadratic control, Common noise, Riccati equation.

\vspace{2mm}\noindent\textbf{AMS subject classification.} 91A16, 49N10, 91A10, 93C41


\section{Introduction}

\section{Introduction}
Given $T\geq 0$. Let $(\Omega,\mathscr{F},\{\mathscr{F}_t\}_{0\leq t\leq T},\mathbb{P})$ be a complete filtered probability space. There are $2N$ independent $d$-dimensional $\mathscr{F}_t$-adapted standard Brownian motions $\{W_t^i,\overline{W}_t^i; 1\leq i\leq N\}$ and an $1$-dimensional standard Brownian motion $W_t^0$. Let $\mathscr{F}_t:=\sigma\{W^i_{s}, \overline W^i_{s}, W^0_{s}, 1\leq i\leq N\}_{0\leq s\leq t}$ augmented by all $\mathbb P$-null set $\mathscr{N}$, which is the full information of this large population (LP) system. Define $\mathscr{F}_t^{i}:= \sigma\{W_s^i,\overline W_s^i, W_s^0\}_{0\leq s\leq t}\vee \mathscr{N}$ and $\mathscr F^i:= \{\mathscr{F}_t^i\}_{0\leq t\leq T}$, which denotes all information of $i$-th agent $\mathcal{A}_i$.

\subsection{Motivation}\label{motivation}

Mean-field games (MFGs) have attracted extensive attention over the past few decades, owing to their profound theoretical implications and broad practical applications (see \cite{BFY2013,ET2015,Carmona2018}). Here, we present an asset-liability management problem with mean-variance performance, which motivated us to study the indefinite MFG problem with partial observation. The market consists of two investable assets. One is a risk-free bond whose price process, denoted by $S^0$, is governed by the following ordinary differential equation (ODE):
\begin{equation*}
    \mathrm{d}S^0_t=r_tS^0_t\mathrm{d}t, \quad S_0^0=s_0.
\end{equation*}
Here $s_0>0$ denotes the initial price and $r_t>0$ represents the risk-free interest rate of this bond. The other is a stock, whose price process $S$ satisfies the following stochastic differential equation (SDE):
\begin{equation*}
    \mathrm{d}S_t=\mu_tS_t\mathrm{d}t+\sigma_tS_t\mathrm{d}W_t^0,\quad S_0=s,
\end{equation*}
where $s>0$ denotes the initial price, $\mu_t>0$ is the appreciation, and $\sigma_t>0$ is the volatility. In the investment market, there are $N$ homogeneous individual investors, each of whom invests through a distinct agency company and receives returns. It is also assumed that there are $N$ agency companies, each aligned with the corresponding investor's objective. Let $\Lambda^i$ denote the total wealth allocated by investor $\mathcal A_i$ through the agency, and let $u^i$ be the amount invested in the stock. Then the remaining amount, $\Lambda^i - u^i$, is invested in the bond. Thus the wealth process $\Lambda^i$ satisfies the following SDE:
\begin{equation*}
    \mathrm{d}\Lambda^i_t=[r_t\Lambda^i_t+(\mu_t-r_t)u^i_t]\mathrm{d}t+\sigma_tu^i_t\mathrm{d}W_t^0,\quad \Lambda_0^i=w,
\end{equation*}
where $w>0$ denotes the initial endowment of the $i$-th investor. In additional, to ensure normal operation fulfill other obligations, each company carries a liability, denoted by $l^i$, which evolves according to the following SDE, \cite{Wang2015}:
\begin{equation*}
    -\mathrm{d}l^i_t=-(r_tl_t^i+b_t)\mathrm{d}t+c_t\mathrm{d}W_t^i+\bar c_t\mathrm{d}\overline W_t^i, \quad l_0^i=l.
\end{equation*}
Here $b_t>0$ denotes the {expected} liability rate, $c$ and $\bar c$ are the volatility of liability, and $l>0$ is the initial liability. For simplicity, we assume that the appreciation rate of the liability equals the risk-free interest rate. Thus, the $i$-th company's cash balance, defined by $x^i=\Lambda^i-l^i$, satisfies
\begin{equation*}
    \mathrm{d}x^i_t=(r_tx_t^i+B_tu_t^i-b_t)\mathrm{d}t+\sigma_tu_t\mathrm{d}W_t^0+c_t\mathrm{d}W_t^i+\bar c_t\mathrm{d}\overline W_t^i,\quad x_0^i=x_0,
\end{equation*}
where $B=\mu-r$ and $x_0=w-l$. Each investor independently determines the amount of wealth invested in the associated company, but cannot access full information about the company's liability. Based on the company's public information, each investor observes a related process $y$, which evolves according to the following SDE:
\begin{equation*}
    \mathrm{d}y_t^i=(G_tx_t^i+\tilde b_t)\mathrm{d}t+\tilde\sigma_t\mathrm{d}\overline W_t^i,\quad y_0^i=0.
\end{equation*}
Additionally, based on the stock price information, each agent also observes a public process $\theta$ {related to the common noise} $W^0$, which satisfies the following dynamics:
\begin{equation*}
    \mathrm{d}\theta=(I_t\theta_t+\check b_t)\mathrm{d}t+\check \sigma_t\mathrm{d}W_t^0,\quad \theta_0=0.
\end{equation*}
It can be seen that when $I=\check b=0$ and $\check\sigma=1$, the common noise becomes directly observable. Let $\mathscr{F}_t^{y^i}:= \sigma\{y_s^i,\theta_s\}_{0\leq s\leq t}$ be the information available to agent $\mathcal A_i$. We then consider a portfolio selection problem under a mean-variance performance criterion, originally introduced by Markowitz \cite{M1952, M1960}. Since Markowitz's seminal work, the mean-variance framework has been extensively studied in the context of continuous-time, multi-period portfolio selection problems, see \cite{LZ2002,ZY2003,XXZ2021,GXZ2023,HSX2023}. This problem is characterized by two conflicting objectives: maximizing the expected terminal wealth and minimizing the associated risk, which is measured by the variance of terminal wealth. In our setting, each agent faces two conflicting objectives. The first is to maximize expected terminal wealth, given by $\mathbb E[x_T^i]$. The second is to minimize the associated risk, measured by $\mathbb E[|x_T^i-x_T^{(N)}|^2]$. Here {$x^{(N)}=\frac{1}{N}\sum_{i=1}^Nx^i$} denotes the average terminal cash balance across all agents. This problem constitutes a multi-objective stochastic optimal control problem, which can be addressed via the following auxiliary control problem (see \cite{Y1974}):

 \textbf{Problem (EX).} Find a strategy $\hat u=(\hat u^1,\hat u^2,\cdots,\hat u^N)$ such that
\begin{equation}\label{cost1}
  \mathcal J_i(\hat u_{\cdot}^i,\hat u_{\cdot}^{-i})=\inf_{u^i}\mathcal J_i(u^i_{\cdot},\hat u_{\cdot}^{-i})=\inf_{u^i}\frac{1}{2}\Big(\gamma \mathbb E\big[ |x_T^i-x_T^{(N)}|^2\big]-\mathbb E[x_T^i]\Big),
\end{equation}
where $\hat u^{-i}=(\hat u^1,\cdots,\hat u^{i-1},\hat u^{i+1},\cdots,\hat u^N)$ and $\gamma>0$ is a constant representing the weight.

Noting the cost functional \eqref{cost1}, Problem (EX) is in fact an indefinite problem. Indeed, the mean-variance portfolio selection problem belongs to a class of indefinite stochastic linear-quadratic (SLQ) optimal control problem, see \cite{ZL2000} for the case without mean-field interaction. {More precisely, this example represents an indefinite problem that fails to meet Condition (PD) but satisfies Condition (RC), (for more details, see Section \ref{sec:application}).} Motivated by above example, we investigate an LP system consisting of $N$ individual agents $\{\mathcal A_i\}_{1\leq i\leq N}$. The dynamics of agent $\mathcal A_i$ are governed by the following SDE:
\begin{equation}\label{xx}
  \left\{
    \begin{aligned}
      &\mathrm{d}x_t^i=\big\{A_tx_t^i+B_tu_t^i+\bar A_tx^{(N)}_t+\bar B_tu^{(N)}_t +b_t\big\} \mathrm{d}t+ \sigma_t \mathrm{d}W_t^i\\
      &\qquad\quad+\big\{D_tx_t^i+F_tu_t^i+\bar D_tx^{(N)}_t+\bar F_tu^{(N)}_t+\bar b_t\big\}\mathrm{d}W_t^0\!+\bar\sigma_t\mathrm{d}\overline W_t^i,\\
      &x_0^i=x,
    \end{aligned}
  \right.
\end{equation}
where $x^{(N)}:=\frac{1}{N}\sum_{i=1}^N x^i$ and $u^{(N)}:= \frac{1}{N}\sum_{i=1}^N u^i$ denote the state-average and control-average across all agents, respectively. We assume that agent $\mathcal A_i$ cannot directly access full information about the state. Instead, she observes a related process $y^i$ described by the following SDE:
\begin{equation}\label{YY}
    \begin{aligned}
      &\mathrm{d}y_t^i=\big\{G_tx_t^i+H_tu_t^i+\bar G_tx^{(N)}_t+\bar H_tu_t^{(N)}+\tilde b_t\big\}
      \mathrm{d}t+\tilde{\sigma}_t\mathrm{d}\overline W_t^i,\qquad
      y_0^i=0.
    \end{aligned}
\end{equation}
We assume that the common noise $W^0$ is observable by all agents through the following shared SDE:
\begin{equation}\label{theta}
    \begin{aligned}
      &\mathrm{d}\theta_t=\big(I_t\theta_t+\check b_t\big)\mathrm{d}t+\check{\sigma}_t \mathrm{d}W_t^0,\qquad
      \theta_0=0.
    \end{aligned}
\end{equation}
{It is worth noting that the introduction of common observation is crucial in partially observable systems where the diffusion term depends on the control, as it ensures the validity of Lemma \ref{lemma24} and Lemma \ref{lemma32} (see Remark \ref{rem+1}). Moreover, if $\check\sigma$ is non-degenerate (see (H1)), it follows from \cite[Theorem 7.16]{Liptser1977} that $\mathcal F^{\theta}=\mathcal F^{W^0}$.}  For convenience, let $u = (u^1, u^2, \cdots, u^N)$ be the set of strategies of all agents and $u^{-i} = (u^1,\cdots, u^{i-1}, u^{i+1},\cdots, u^N )$ be the set of strategies expect for $i$-th agent. Then, the cost functional of $i$-th agent takes the following form
\begin{equation}\label{cost}
  \begin{aligned}
    &\mathcal J_i(u^i_{\cdot},u^{-i}_{\cdot})=\frac{1}{2}\mathbb{E}\bigg[\int_0^T\langle Q_t\big(x_t^i -\alpha_1 x^{(N)}_t\big)+2q_t,x_t^i-\alpha_1 x^{(N)}_t\rangle +\langle R_t\big(u_t^i-\beta_1 u^{(N)}_t\big)+2r_t, u_t^i\\
    &\qquad-\beta_1 u^{(N)}_t\rangle
    +2\langle S_t\big(u_t^i-\beta_2 u^{(N)}_t\big), x_t^i-\alpha_2x^{(N)}_t\rangle\mathrm{d}t +\langle L_T\big(x_T^i-\alpha_3 x^{(N)}_T\big)+2l_T, x_T^i-\alpha_3 x^{(N)}_T\rangle\bigg],
  \end{aligned}
\end{equation}
{where $q$, $r$ and $l_T$ correspond to the coefficients of the first-order (linear) terms. They capture the linear dependence of the cost functional on the state and control variables.} We aim to find a Nash equilibrium $\hat u=(\hat u^{1},\hat u^{2},\cdots,\hat u^{N})$ such that $\mathcal J_i(\hat u^i_{\cdot},\hat u^{-i}_{\cdot})=\inf_{u^i}\mathcal J_i(u^i_{\cdot},\hat u^{-i}_{\cdot})$.

\subsection{Literature review and contributions}
Unlike traditional control and game systems, MFG models involve a large number of participants, whose individual actions are negligible, yet whose collective behavior can significantly influence the overall system or environment. The MFG framework was originally developed by Huang, Caines and Malham\'e \cite{HMC2007}, and independently by Lasry and Lions \cite{LL2007}. A widely used approach for solving MFG problems is to construct an approximate Nash equilibrium via an associated auxiliary control problem, derived by analyzing the limiting behavior of the system, see \cite{Hu20182,Xu2020,Wang2020}. Furthermore, there exists a substantial body of literature on MFG problems with common noise, see e.g. \cite{CDL2016,KT2019}.

The linear-quadratic (LQ) optimal control problem is a fundamental topic in control theory. It is well known that in the deterministic LQ setting, the control weighting matrix in the cost functional must be positive definite, see e.g. \cite{K1960,AM2007}. However, this assumption can be relaxed in the stochastic LQ setting, where the weighting matrix may be zero or even negative. The indefinite LQ control problem was first investigated through the solvability of the associated Riccati equations \cite{CLZ1998,CZ2000}, and was subsequently extended by other researchers \cite{LZ1999,RMZ2002}. Yu \cite{Y2013} proposed an equivalent cost functional method that transforms indefinite LQ problems into standard ones, this method was further developed by \cite{LLY2020}. Xu and Zhang \cite{Xu2020}, as well as Wang, Zhang and Zhang \cite{Wang20202} studied MFG with indefinite control weighting by imposing relatively strong technical conditions.

In the aforementioned literature, participants are assumed to have full knowledge of the system state. However, in practice, participants typically make decisions based on partial observations of the state. There is a substantial body of literature on stochastic control problems with partial observation, which can be broadly classified into two categories. The first class assumes that the observation process is an uncontrolled Brownian motion. In such cases, the control problem can often be handled via Girsanov's transformation (see \cite{LT1995,T1998}). The second class considers the observation process as a controlled stochastic process. In this paper, we focus on the second case. In such settings, the control is adapted to the observation filtration, resulting in a circular dependency between the control and the observation process. Wonham \cite{W1968} proposed the separation principle to address this issue. This principle allows to first compute the filtering of state, and then to solve fully observed optimal control problems driven by the filtering states. However, in many cases, the mean square error of the state estimate still depends on the control, rendering the Wonham separation principle inapplicable. Wang and Wu \cite{WW2008} introduced a backward separation approach for partially observed LQ control problem by first decomposing the state and observation, and then computing the filtering, see \cite{Wang2015} for more details. Recently, partially observed MFGs have also been studied, see \cite{BFH2021,CDW2024}. {Compared with \cite{WW2008,Wang2015,BFH2021,CDW2024}, this paper extends the backward separation approach to settings where the diffusion term depends on control variables, and further studies a class of MFGs with indefinite weighting matrix in the cost functional. For better illustration, we provide the following comparison table.\\

\begin{tabular}{|c|c|c|c|}
\hline
Literature & Condition (PD) & Control or game & Common noise \\
\hline
Wang \& Wu \cite{WW2008} & Satisfied & Control & No\\
\hline
Wang, Wu \& Xiong \cite{Wang2015} & Satisfied & Control & No\\
\hline
Bensoussan, Feng \& Huang \cite{BFH2021} & Satisfied & Mean-field game& Yes\\
\hline
Chen, Du \& Wu \cite{CDW2024} & Satisfied & Mean-field game& No\\
\hline
This paper & \makecell{Not satisfied\\ (Indefinite)} & Mean-field game & Yes\\
\hline
\end{tabular}}\\

{The study of the indefinite control} problems has primarily focused on stochastic LQ systems. To solve this problem, the control variable in the state's diffusion term plays a key role. However, the drift term of the observation process grows linearly rather than being uniformly bounded, which makes the Girsanov theorem difficult to verify (see \cite{LT1995,T1998}). The classical backward separation approach fails to handle cases where the control variable appears in the state's diffusion term (see \cite{Wang2015}). As a result, the indefinite partially observed LQ control problem has long remained open. Fortunately, by introducing a common noise $W^0$ and an associated observation process $\theta$, we extend the backward separation approach to the stochastic systems with control variables in the state's diffusion term related to common noise. Consequently, the indefinite partially observed LQ control problem is resolved within this framework.

The main contributions of this paper can be summarized as follows,

(1) A class of indefinite partially observed MFG problem with common noise is studied. Each agent's state is governed by a partially observed SDE, where the diffusion term depends on the state, control, state-average $x^{(N)}$ and control-average $u^{(N)}$ of all agents. Notably, the appearance of common noise term allows that the weight matrices in the cost functional can be indefinite. It looks that our paper is the first one to study indefinite partially observed control problem.

(2) In all existing literature on partially observed LQ problem (see, e.g. \cite{WW2008,Wang2015,BFH2021,CDW2024}), the control variable cannot enter the diffusion term of state. Fortunately, we have addressed this limitation. By introducing a common observation process $\theta$, our model allows the control variable to enter the diffusion term related common noise. Then we extend the backward separation approach to partially observed control problems of such complex systems.

(3) In the investigation of indefinite control problems, studying the solvability of Riccati equation (see e.g. \cite{CLZ1998,CZ2000,RMZ2002,Xu2020}) or assuming the uniformly convex condition (see e.g. \cite{SLY2016,Wang20202,SXY2021}) are the main methods. Inspired by \cite{RMZ2002,Y2013}, we propose a novel approach to solve the indefinite LQ partially observed MFG by using a relaxed compensator through a flexible condition (Condition (RC)), which can be easily verified. The existence of relaxed compensator can imply the solvability of indefinite Riccati equation (see Theorem \ref{thm8314}) and the uniform convexity of the cost functional w.r.t. control (see Lemma \ref{rmk031337}). Moreover, by virtue of relaxed compensator and linear transformation, we establish the well-posedness of the Hamiltonian system \eqref{Hamiltonian}, which is an FBSDE that does not satisfy the monotonicity condition.

(4) The decentralized strategy has been proved to be an $\varepsilon$-Nash equilibrium. Inspired by the method of equivalent cost functional, we show that the indefinite cost functional of origin problem is equivalent to a standard cost functional. Furthermore, we obtain a useful inequality \eqref{important} (the boundedness of alternative control in the sense of $L^2$) which plays a key role in proving $\varepsilon$-Nash equilibrium without imposing additional assumption as in \cite{Xu2020,Wang20202}.

\subsection{Notations and terminology}

We denote the $m$-dimensional Euclidean space by $\mathbb{R}^m$ with norm  $|\cdot|$ and inner product  $\langle \cdot,\cdot\rangle$. $D^{\top}$ (resp. $D^{-1}$) denotes the transposition (resp. inverse) of $D$. $\mathbb S^m$ denotes the set of symmetric $m\times m$ matrices with real elements. If $D\in \mathbb{S}^n$ is positive definite (positive semi-definite), we write $D>(\geq)\ 0$. {Moreover, if an $\mathbb S^n$-valued deterministic function $D$ is uniformly positive definite, i.e. there exists $\lambda_0>0$ such that $D_t\geq \lambda_0 I_n$ for every $t\in [0,T]$, we write $D\gg 0$.} For a given Hilbert space $\mathbb H$, $L^2_{\mathscr{F}_T}(\mathbb{H})$ denotes the space of all $\mathbb H$-valued ${\mathscr{F}_T}$-measurable, square-integrable random variables; $L^{\infty}([0,T]; \mathbb{H})$ denotes the space of all $\mathbb H$-valued deterministic uniformly bounded functions; $C([0,T];\mathbb H)$ denotes the space of all $\mathbb H$-valued deterministic functions $\phi$ such that $\dot\phi$ is continuous; $L^2_{\mathscr{F}}([0,T];\mathbb{H})$ denotes the space of all $\mathbb H$-valued, $\mathscr{F}_t$-adapted, square-integrable processes; $S^2_{\mathscr{F}}([0,T];\mathbb{H})$ denotes the space of all $\mathbb H$-valued, $\mathscr{F}_t$-adapted continuous processes $\phi$ such that $ \mathbb{E}[\sup_{0\leq t\leq T}|\phi_t|^2] < \infty$. Let $\mathcal M_{\mathscr F}:=L^2_{\mathscr{F}}([0,T];\mathbb{H})\times L^2_{\mathscr{F}}([0,T];\mathbb{H})\times L^2_{\mathscr{F}}([0,T];\mathbb{H})$, $\bar{\mathcal M}_{\mathscr F}:=S^2_{\mathscr{F}}([0,T];\mathbb{H})\times \mathcal M_{\mathscr F}(\mathbb H)$.

The remaining sections are organized as follows. Section \ref{sec:problem} formulates the indefinite LQ partially observed MFG problem with common noise. In section \ref{sec:the}, we obtain the decentralized strategies using the backward separation approach and Hamiltonian approach. By virtue of Riccati equation, we derive the feedback representation of the decentralized strategies. The corresponding $\varepsilon$-Nash equilibrium has been verified in section \ref{sec:varepsilon}. In section \ref{sec:application}, we solve a mean-variance portfolio selection problem raised at the beginning of this paper.

\section{Problem Formulation}\label{sec:problem}
In this section, we would like to characterize the MFG problem proposed in Section \ref{motivation} more accurately. We define the observable filtration $\mathscr{F}_t^{y^i}:=\sigma\{y_s^i,\theta_s\}_{0\leq s\leq t}$ of $i$-th agent $\mathcal{A}_i$; $\mathscr{F}_t^y:=\bigvee_{i=1}^N \mathscr{F}_t^{y^i}$ denotes all observed information of LP system; $\mathscr{F}^{\theta}_t:=\sigma\{\theta_s\}_{0\leq s\leq t}$ denotes the information of common observation. For each agent, her strategy may be $\mathscr{F}_t^y$-adapted, which is the so-called centralized strategy. Note that the individual observation $y^i$ is a controlled process, then the circular dependence between the observation and control arises. Thus we need to solve this circular dependence. Let us define $x^{i,0}$ and $y^{i,0}$, for $i=1,2,\cdots, N$, by
\begin{equation}\label{x1}
    \begin{aligned}
      &\mathrm{d}x_t^{i,0}=\big(A_tx_t^{i,0}+\bar A_tx^{(N)}_{0,t}\big)\mathrm{d}t
      +\big(D_tx_t^{i,0}+ \bar D_tx^{(N)}_{0,t}\big)\mathrm{d}W_t^0 +\sigma_t\mathrm{d}W_t^i+\bar\sigma_t\mathrm{d}\overline W_t^i,\quad x_0^{i,0}=x,
    \end{aligned}
\end{equation}
\begin{equation}\label{Y1}
    \begin{aligned}
      \mathrm{d}y_t^{i,0}=\big(G_tx_t^{i,0}+\bar G_tx^{(N)}_{0,t}\big)\mathrm{d}t +\tilde{\sigma}_t\mathrm{d}\overline W_t^i,\quad
      y_0^{i,0}=0,
    \end{aligned}
\end{equation}
where $x^{(N)}_{0}:= \frac{1}{N}\sum_{i=1}^N x^{i,0}$. Let $u^i\in L^2_{\mathscr F}([0,T];\mathbb R^m)$ be a control process, define $x^{i,1}$ and $y^{i,1}$ by
\begin{equation}\label{x2}
    \begin{aligned}
      &\mathrm{d}{x}^{i,1}_t=\big(A_tx_t^{i,1}+B_tu_t^i+\bar A_tx^{(N)}_{1,t}+\bar B_tu^{(N)}_t+b_t\big)\mathrm{d}t\\
&\qquad\qquad+\big(D_tx_t^{i,1}+F_tu_t^i+\bar D_tx^{(N)}_{1,t}+\bar F_tu^{(N)}_t+\bar b_t\big)\mathrm{d}W_t^0,\quad x^{i,1}_0=0,
    \end{aligned}
\end{equation}
\begin{equation}\label{Y2}
    \begin{aligned}
      &\mathrm{d}y_t^{i,1}=\big(G_tx_t^{i,1}+H_tu_t^i+\bar G_tx^{(N)}_{1,t}+\bar H_tu^{(N)}_t+\tilde b_t\big)\mathrm{d}t,\quad
      y_0^{i,1}=0,
    \end{aligned}
\end{equation}
where $x^{(N)}_{1} := \frac{1}{N}\sum_{i=1}^N x^{i,1}$. We give the following assumption on the coefficients.

\textbf{(H1)}
   $A, \bar A, D, \bar D, G, \bar G\in L^{\infty}([0,T]; \mathbb{R}^{n\times n})$, $b, \bar b, \tilde{b}\in L^{\infty}([0,T];\mathbb{R}^{n})$, $\sigma, \bar\sigma, \tilde\sigma$ $\in L^2([0,T];\mathbb R^{n\times d})$, $B, \bar B,  F, \bar F, H, \bar H \in L^{\infty}([0,T];\mathbb{R}^{n\times m})$, $I, \check b, \check\sigma \in L^{\infty}([0,T]; \mathbb{R})$, $\check\sigma$ is non-degenerate, $x$ is a constant.

Under (H1), system \eqref{x1}-\eqref{Y2} admits a unique solution. We define $x^i=x^{i,0}+x^{i,1}$ and $y^i=y^{i,0}+y^{i,1}$. It is easy to check that $x^i$ (resp. $y^i$) is the unique solution of \eqref{xx} (resp. \eqref{YY}), and $x^{(N)}=x_{0}^{(N)}+x_{1}^{(N)}$. We also denote ${\mathscr{F}}_t^{y^0}=\sigma\{y_s^{i,0},\theta_s; 1\leq i\leq N\}_{0\leq s\leq t}$. To overcome circular dependency, we give the following set of strategies,
\begin{equation*}
  \begin{aligned}
    \mathscr{U}_{c}^{i,0}=\bigg\{u^i\ |\ u_t^i \text{ is an } {\mathscr{F}}_t^{y^0}&\text{-adapted process valued in } \mathbb{R}^m,\text{ such that } \mathbb{E}\bigg[\sup_{0\leq t\leq T}|u_t^i|^2 \bigg]<\infty \bigg\}.
  \end{aligned}
\end{equation*}

\begin{definition}\label{Def21}
  Define the admissible centralized strategy set $\mathscr{U}_{c}^{i}$ as the set of all controls $u^i$ satisfying $u^i\in \mathscr{U}_{c}^{i,0}$ and $u^i$ is $\mathscr{F}^y$-adapted.
\end{definition}
Then we have
\begin{lemma}\label{lemma24}
  For any $u^i\in \mathscr{U}_{c}^{i}$, it holds that $\mathscr{F}_t^y={\mathscr{F}}_t^{y^0}$.
\end{lemma}
\begin{proof}
  For any $u^i\in \mathscr{U}_{c}^{i}$, we know that $u_t^i$ is ${\mathscr{F}}_t^{y^0}$-adapted. Then $x_t^{i,1}$ is ${\mathscr{F}}_t^{y^0}$-adapted by noticing \eqref{x2}, thus $y_t^{i,1}$ is also ${\mathscr{F}}_t^{y^0}$-adapted by \eqref{Y2}. Then $y^i=y^{i,0}+y^{i,1}$ is ${\mathscr{F}}_t^{y^0}$-adapted, that is $\mathscr{F}_t^y\subseteq {\mathscr{F}}_t^{y^0}$. According to the similar argument, we can obtain $ {\mathscr{F}}_t^{y^0}\subseteq \mathscr{F}_t^y$ via the equality $y^{i,0}=y^i-y^{i,1}$.
\end{proof}
{\begin{remark}\label{rem+1}
    To ensure the solvability of indefinite stochastic LQ control or game problems, the diffusion term in the state equation must involve the control variable, see \cite{CLZ1998,CZ2000}. However, existing theories of partially observed control problem cannot handle this situation, see \cite{WWX2018} . In fact, the linear decomposition \eqref{x1}-\eqref{Y2} can no longer decouple the control $u^i$ from the common noise $W^0$ in the control-dependent state \eqref{x2}. As a result, it becomes impossible to establish a connection between the filtrations $\mathscr F^y$ and $\mathscr F^{y^0}$. By introducing a common observation process $\theta$ for the common noise $W^0$, we propose a suitable definition of the strategy sets $\mathcal U_c^{i,0}$ and $\mathcal U_c^i$, under which Lemma \ref{lemma24} can be established within the $\mathscr F^y$. Moreover, a similar property holds for Lemma \ref{lemma32},  which serves a foundation step in the proof of Lemma \ref{lemma332}.
\end{remark}}

Let us recall the cost functional \eqref{cost} and introduce the following assumption.

\textbf{(H2)}
   $Q \in L^{\infty}([0,T];\mathbb{S}^n)$, $R \in L^{\infty}([0,T];\mathbb{S}^m)$, $S \in L^{\infty}([0,T]; \mathbb R^{n\times m})$, $q\in l^{\infty}([0,T];\mathbb R^n)$, $r\in L^{\infty}([0,T];\mathbb R^m)$, $L_T\in \mathbb{S}^n$, $l_T\in \mathbb R^n$, $\alpha_1, \alpha_2, \alpha_3, \beta_1, \beta_2$ are constants.

\begin{remark}
  Obviously,  for any given $x\in \mathbb R^n$ and any admissible strategy $u=(u^1,u^2,\cdots,u^N)$, the cost functional \eqref{cost} is well-defined under (H1)-(H2). It is worth pointing out that we do not impose any positive-definiteness/non-negativeness conditions on $Q$, $R$ and $L_T$. It looks like our paper is the first one to study the indefinite partially observed stochastic control problems.
\end{remark}

A basic solution for \eqref{xx}-\eqref{cost} is a Nash equilibrium $\hat u=(\hat u^{1},\hat u^{2},\cdots,\hat u^{N})$, where $u^i\in \mathcal U_c^i$, for each $1\leq i\leq N$. However, such a solution is impractical when the LP system consists of a large number of agents, due to the prohibitive computational complexity and unrealistic information requirements. Hence, we aim to establish the $\varepsilon$-Nash equilibrium.

\section{The Limiting Control Problem}\label{sec:the}

To design the decentralized strategies, we need to study the associated limiting problem when the agent number $N$ tends to infinity. Suppose that $(x^{(N)},u^{(N)})$ are approximated by $(x^0,u^0)$, and here $(x^0,u^0) $ is some $\mathscr{F}^{\theta}$-adapted process pair which will {be defined later (see \eqref{px0}-\eqref{92133})}. We also assume that $x^{(N)}_{0}$ and $x^{(N)}_{1}$ are respectively approximated by $x^{0,0}$ and $x^{1,0}$ with $x^{0,0}+x^{1,0}=x^0$. Then we introduce the following auxiliary limiting state,
\begin{equation}\label{lxx}
    \begin{aligned}
      \mathrm{d}X_t^i&=\big(A_tX_t^i+B_tu_t^i+\bar A_tx_t^0+\bar B_tu_t^0+b_t\big) \mathrm{d}t +\sigma_t\mathrm{d}W_t^i\\
      &\quad +\big(D_tX_t^i+F_tu_t^i+\bar D_tx_t^0+\bar F_tu_t^0+\bar b_t\big)\mathrm{d}W_t^0+ \bar \sigma_t\mathrm{d}\overline W_t^i,\quad
      X_0^i=x,
    \end{aligned}
\end{equation}
and the limiting individual observation process,
\begin{equation}\label{lYY}
    \begin{aligned}
      &\mathrm{d}Y_t^i=\big(G_tX_t^i+H_tu_t^i+\bar H_tx_t^0+\bar G_t u_t^0+\tilde b_t\big)
      \mathrm{d}t+\tilde{\sigma}_t\mathrm{d}\overline W_t^i,\quad
      Y_0^i=0.
    \end{aligned}
\end{equation}
{Moreover, the common observation process $\theta$ still satisfies \eqref{theta}, that is,
\begin{equation*}
    \mathrm{d}\theta=(I_t\theta_t+\check b_t)\mathrm{d}t+\check \sigma_t\mathrm{d}W_t^0,\quad \theta_0=0.
\end{equation*}}
The above limiting state and limiting individual observation can both be decomposed as two parts, as in the following arguments.
Let $X^{i,0}$ and $Y^{i,0}$ be respectively given by
\begin{equation}\label{lx1}
    \mathrm{d}X_t^{i,0}=(A_tX_t^{i,0}+\bar A_tx^{0,0}_t)\mathrm{d}t  +(D_tX_t^{i,0}+ {\color{blue}\bar D_t}x^{0,0}_t)\mathrm{d}W_t^0 +\sigma_t\mathrm{d}W_t^i+\bar\sigma_t\mathrm{d}\overline W_t^i,\quad X_0^{i,0}=x,
\end{equation}
and
\begin{equation}\label{lY1}
    \mathrm{d}Y_t^{i,0}=\big(G_tX_t^{i,0}+\bar G_tx^{0,0}_t\big)\mathrm{d}t +\tilde{\sigma}_t\mathrm{d}\overline W_t^i,\quad       Y_0^{i,0}=0.
\end{equation}
Let $X^{i,1}$ and $Y^{i,1}$ be the solutions of
\begin{equation}\label{lx2}
    \begin{aligned}
        &\mathrm{d}{X}^{i,1}_t=\big(A_tX_t^{i,1}+B_tu_t^i+\bar A_tx^{1,0}_t+\bar B_tu_t^0+b_t\big)\mathrm{d}t \\
      &\qquad\quad+\big(D_tX_t^{i,1}+F_tu_t^i+\bar D_tx^{1,0}_t+\bar F_tu_t^0+\bar b_t\big)\mathrm{d}W_t^0,\quad
      X^{i,1}_0=0,
    \end{aligned}
\end{equation}
and
\begin{equation}\label{lY2}
    \mathrm{d}Y_t^{i,1}=\big(G_tX_t^{i,1}+H_tu_t^i+\bar G_tx^{1,0}_t+\bar H_tu_t^0+\tilde b_t\big)\mathrm{d}t,\quad
      Y_0^{i,1}=0,
\end{equation}
respectively.
We then have $X^{i,0}+X^{i,1}=X^i$ and $Y^{i,0}+Y^{i,1}=Y^i$.

Next, we aim to design the decentralized strategies. We mention that there is a circular dependency between the limiting observation process $Y^i$ and the control strategy $u^i$. To overcome this difficulty, we set $\mathscr{F}^{Y^i}_t=\sigma\{Y_s^i,\theta_s\}_{0\leq s\leq t}$ and $\mathscr{F}^{Y^{i,0}}_t=\sigma\{Y_s^{i,0},\theta_s\}_{0\leq s\leq t}$ and define
\begin{equation*}
  \begin{aligned}
    \mathscr{U}_{d}^{i,0}=\bigg\{u^i\ |\ u_t^i \text{ is an } {\mathscr{F}}_t^{Y^{i,0}}&\text{-adapted process valued in } \mathbb{R}^m,\text{ such that } \mathbb{E}\bigg[\sup_{0\leq t\leq T}|u_t^i|^2  \bigg]<\infty \bigg\}.
  \end{aligned}
\end{equation*}
\begin{definition}\label{def2131}
  Define admissible decentralized strategy set $\mathscr{U}_{d}^{i}$ as the set of all controls $u^i$ satisfying $u^i\in \mathscr{U}_{d}^{i,0}$ and $u^i$ is $\mathscr{F}^{Y^i}$-adapted.
\end{definition}

The associated limiting cost functional becomes
\begin{equation}\label{lcost}
  \begin{aligned}
     J_i(u^i_{\cdot})&=\frac{1}{2}\mathbb{E}\bigg[\int_0^T\big\{\langle Q_t(X_t^i -\alpha_1 x_t^0)+2q_t,X_t^i-\alpha_1 x_t^0\rangle +\langle R_t(u_t^i-\beta_1u_t^0)+2r_t, u_t^i-\beta_1u_t^0\rangle\\
    &\qquad+2\langle S_t(u_t^i-\beta_2u_t^0), X_t^i-\alpha_2 x_t^0\rangle\big\}\mathrm{d}t+\langle L_T(X_T^i-\alpha_3 x_T^0)+ 2l_T, X_T^i-\alpha_3 x_T^0\rangle\bigg],
  \end{aligned}
\end{equation}
and the auxiliary partially observed LQ problem can be formulated as follows:

\textbf{Problem (MFD).}
  For the $i$-th agent $\mathcal{A}_i$, $i=1,2,\cdots,N$, {find a priori strategy $\bar{u}^i\in \mathscr{U}_{d}^{i}$, depending on the parameter $(x^0,u^0)$,  such that}
  \begin{equation*}
    J_i(\bar{u}_{\cdot}^i)=\inf_{u^i\in \mathscr{U}_{d}^{i}}J_i(u_{\cdot}^i).
  \end{equation*}

Problem (MFD) is called well-posed if the infimum of $J_i(u_{\cdot}^i)$ is finite. If Problem (MFD) is well-posed and the infimum of the cost functional can be achieved, then Problem (MFD) is said to be solvable. Any $\bar u^i$ satisfies $J_i(\bar u_{\cdot}^i)=\inf_{u^i\in \mathscr U_{d}^i} J_i(u^i_{\cdot})$ is called an optimal control of Problem (MFD), and related $\bar X^i$ (see \eqref{lxx}) and $J_i(\bar u_{\cdot}^i)$ (see \eqref{lcost}) are called the optimal state and the optimal cost functional, respectively. Then $(\bar X^i, \bar u^i)$ is called an optimal pair of Problem (MFD). {Moreover, let $\bar Y^i$ be the optimal observation related to $\bar u^i$. For a given stochastic process $\Phi$,
\begin{equation}\label{filteringsymbol}
    \hat{\Phi}=\mathbb E[\Phi|\mathcal F^{\bar Y^i}], \qquad \text{and} \qquad  \check \Phi=\mathbb E[\Phi|\mathscr F^{\theta}],
\end{equation}
respectively denote the optimal filtering of $\Phi$ with respect to the filtration $\mathscr F^{\bar Y^i}$ and $\mathscr F^{\theta}$, where $\mathscr F^{\bar Y^i}_t=\sigma\{\bar Y^i_s,\theta_s\}_{0\leq s\leq t}$ and $\mathscr F^{\theta}_t=\sigma\{\theta_s\}_{0\leq s\leq t}$.}  Similar to Lemma \ref{lemma24}, we obtain
\begin{lemma}\label{lemma32}
  For any $u^i\in \mathscr{U}_{d}^{i}$, $\mathscr{F}_t^{Y^i}=\mathscr{F}_t^{Y^{i,0}}$.
\end{lemma}

Let us give a useful lemma, which implies that we can find an optimal strategy $u^i\in \mathscr{U}_{d}^{i,0}$ instead of $u^i\in \mathscr{U}_{d}^{i}$ to minimize $J_i$. The proof one can see Appendix \ref{Appendix A}.
\begin{lemma}\label{lemma332}
  Under (H1)-(H2), we have
  \begin{equation*}
    \inf_{u^i\in \mathscr{U}_{d}^{i}}{J}_i(u^i_{\cdot})=\inf_{u^i\in \mathscr{U}_{d}^{i,0}}{J}_i(u^i_{\cdot})
  \end{equation*}
\end{lemma}

\subsection{Optimal decentralized strategy}
In this subsection, we could get the optimal strategy of Problem (MFD) by virtue of the stochastic maximum principle.

\begin{proposition}\label{Thm34}
  Let (H1)-(H2) hold. The admissible strategy $\bar{u}^i\in \mathscr{U}_{d}^{i}$ is an optimal decentralized strategy of Problem (MFD) if and only if the following stationary condition holds:
  \begin{equation}\label{optimalcontrol}
    B^{\top}_t\mathbb E\big[\varphi_t^i\big|\mathscr{F}_t^{\bar{Y}^{i}}\big] +F^{\top}_t\mathbb E\big[\eta_t^i\big|\mathscr{F}_t^{\bar{Y}^{i}}\big]+ S_t\mathbb E\big[\bar X_t^i-\alpha_2 x_t^0\big|\mathscr{F}_t^{\bar{Y}^{i}}\big] +R_t(\bar u_t^i-\beta_1u_t^0)+r_t=0,
\end{equation}
  and the following convexity condition holds for any $v^i\in \mathscr{U}_{d}^{i}$,
  \begin{equation}\label{92121}
    \mathbb E\bigg[\int_0^T\big\{\langle Q_t \tilde X_t^i,\tilde X_t^i\rangle +2\langle S_tv^i_t, \tilde X_t^i\rangle +\langle R_tv_t^i,v_t^i\rangle\big\}\mathrm{d}t+\langle L_T\tilde X_T^i,\tilde X_T^i\rangle\bigg]\geq 0.
  \end{equation}
  Here $\tilde X^i$ solves the following SDE,
  \begin{equation}\label{variational}
    \begin{aligned}
      &\mathrm{d}\tilde X_t^i=\big(A_t\tilde X_t^i+B_tv_t^i\big) \mathrm{d}t  +\big(D_t\tilde X_t^i+F_tv_t^i\big)\mathrm{d}W_t^0,\quad
      \tilde X_0^i=0.
    \end{aligned}
\end{equation}
and $(\varphi^i,\eta^i,\zeta^i,\vartheta^i)\in \mathcal M_{\mathscr F^i}$ solves the following BSDE,
  \begin{equation}\label{adjoint}
 \left\{
    \begin{aligned}
      &\mathrm{d}\varphi_t^i=-\big\{A^{\top}_t\varphi_t^i+D^{\top}_t\eta_t^i+Q_t(\bar X_t^i- \alpha_1 x_t^0)+S_t(\bar u_t^i-\beta_2 u_t^0)+ q_t\big\}\mathrm{d}t+\eta_t^i\mathrm{d}W_t^0 +\zeta_t^i\mathrm{d} W_t^i+\vartheta_t^i\mathrm{d}\overline W_t^i,\\
      &\varphi_T^i=L_T(\bar X_T^i-\alpha_3 x_T^0)+l_T,
    \end{aligned}
  \right.
\end{equation}
\end{proposition}
\begin{proof}
If $\bar{u}^i$ is an optimal strategy of Problem (MFD), Lemma \ref{lemma332} yields that ${J}_i(\bar{u}_{\cdot}^i)= \inf_{u^i\in\mathscr{U}_{d}^{i,0}}{J}_i(u_{\cdot}^i)$. For any $v^i\in \mathscr{U}_{d}^{i,0}$, define $X^{i,\epsilon}$ be the solution of following SDE with $u^{i,\epsilon}:= \bar{u}^i+\epsilon v^i\in \mathscr{U}_{d}^{i,0}$, $0< \epsilon < 1$,
\begin{equation*}
    \begin{aligned}
      \mathrm{d}X^{i,\epsilon}_t&=\big(A_tX^{i,\epsilon}_t+B_tu^{i,\epsilon}_t+\bar A_tx_t^0+\bar B_t u_t^0+b_t\big) \mathrm{d}t +\sigma_t\mathrm{d}W_t^i\\
      &\quad +\big(D_tX^{i,\epsilon}_t+F_tu^{i,\epsilon}_t+\bar D_tx_t^0+\bar F_tu_t^0+\bar b_t\big)\mathrm{d}W_t^0+\bar \sigma_t\mathrm{d}\overline W_t^i,\quad
      X^{i,\epsilon}_0=x.
    \end{aligned}
\end{equation*}
Then one can check that $\tilde X^i:= \frac{X^{i,\epsilon}-\bar X^i}{\epsilon}$ is independent of $\epsilon$ and satisfies \eqref{variational}. Applying It\^o's formula to $\langle \varphi_t^i,\tilde X_t^i\rangle$, then it follows that
\begin{equation*}
\begin{aligned}
  J_i(u_{\cdot}^{i,\epsilon})=J_i(\bar u_{\cdot}^i)&+\epsilon \mathbb E\bigg[\int_0^T\big\{\langle B^{\top}_t\varphi_t^i +F^{\top}_t\eta_t^i+ S_t^{\top}(\bar X_t^i-\alpha_2 x_t^0)+ R_t(\bar u_t^i-\beta_1 u_t^0) +r_t,v_t^i\rangle\big\} \mathrm{d}t\bigg]\\
  &+\frac{\epsilon^2}{2}\mathbb E\bigg[\int_0^T\big\{\langle Q_t \tilde X_t^i,\tilde X_t^i\rangle +2\langle S_tv^i_t, \tilde X_t^i\rangle +\langle R_tv_t^i,v_t^i\rangle\big\}\mathrm{d}t+\langle L_T\tilde X_T^i,\tilde X_T^i\rangle\bigg].
\end{aligned}
\end{equation*}
From $J_i(u^{i,\epsilon}_{\cdot})\geq J_i(\bar u_{\cdot}^i)$, we have that \eqref{92121} holds and
\begin{equation*}
  \mathbb E\bigg[\int_0^T\Big\{\langle B^{\top}_t\varphi_t^i +F^{\top}_t\eta_t^i+ S_t^{\top}(\bar X_t^i-\alpha_2 x_t^0)+ R_t(\bar u_t^i-\beta_2 u_t^0) +r_t,v_t^i\rangle\Big\} \mathrm{d}t\bigg]=0,
\end{equation*}
which yields that,
\begin{equation*}
    B^{\top}_t\mathbb E\big[\varphi_t^i\big|\mathscr{F}_t^{\bar{Y}^{i,0}}\big] +F^{\top}_t\mathbb E\big[\eta_t^i\big|\mathscr{F}_t^{\bar{Y}^{i,0}}\big]+ S_t\mathbb E\big[\bar X_t^i-\alpha_2 x_t^0\big|\mathscr{F}_t^{\bar{Y}^{i,0}}\big] +R_t(\bar u_t^i-\beta_1u_t^0)+r_t=0.
\end{equation*}
Noticing $\bar u^i\in \mathscr{U}_{d}^{i}$, we have $\mathscr{F}^{\bar Y^{i,0}}_t=\mathscr{F}^{\bar Y^i}_t$ by Lemma \ref{lemma32}, then we obtain \eqref{optimalcontrol}.

In addition, for any given $v^i\in \mathcal U_{d}^i$, we can easily check that the difference between $J_i(v^i_{\cdot})$ and $J_i(\bar u^i_{\cdot})$ are $J_i(v^i_{\cdot})-J_i(\bar u^i_{\cdot})\geq 0$, which implies $\bar u^i$ given by \eqref{optimalcontrol} is an optimal control.
\end{proof}

Now, combining \eqref{lxx}, \eqref{optimalcontrol} and \eqref{adjoint}, recalling notations \eqref{filteringsymbol}, we obtain the  Hamiltonian system for agent $\mathcal{A}_i$:
\begin{equation}\label{Hamiltonian}
  \left\{
    \begin{aligned}
      &\mathrm{d}\bar X_t^i\!=\!\big(A_t\bar X_t^i\!+\!B_t\bar u_t^i\!+\!\bar A_tx_t^0\!+\!\bar B_tu_t^0\!+\!b_t\big) \mathrm{d}t
        \!+\!\big(D_t\bar X_t^i\!+\!F_t\bar u_t^i\!+\!\bar D_tx_t^0\!+\!\bar F_tu_t^0\!+\!\bar b_t\big)\mathrm{d}W_t^0\!+\!\sigma_t\mathrm{d}W_t^i\!+\!\bar \sigma_t\mathrm{d}\overline W_t^i,\\
      &\mathrm{d}\varphi_t^i\!=\!-\big[A^{\top}_t\varphi_t^i\!+\!D^{\top}_t\eta_t^i\!+\!Q_t(\bar X_t^i\!-\! \alpha_1 x_t^0)\!+\!S_t(\bar u_t^i\!-\!\beta_2 u_t^0)\!+\! q_t\big]\mathrm{d}t \!+\!\eta_t^i\mathrm{d}W_t^0 \!+\!\zeta_t^i\mathrm{d} W_t^i\!+\!\vartheta_t^i\mathrm{d}\overline W_t^i,\\
      &B^{\top}_t\hat\varphi_t^i \!+\!F^{\top}_t\hat\eta_t^i\!+\! S_t(\hat{\bar X}_t^i\!-\!\alpha_2 x_t^0) \!+\!R_t(\bar u_t^i\!-\!\beta_1u_t^0)\!+\!r_t\!=\!0,\\
      &\bar X_0^i\!=\!x,\quad \varphi_T^i\!=\!L_T(\bar X_T^i\!-\!\alpha_3 x_T^0)\!+\!l_T.
    \end{aligned}
  \right.
\end{equation}

{Since the quadruple $(Q, S, R, L_T)$ is indefinite, Hamiltonian system \eqref{Hamiltonian} no longer satisfies the monotonicity condition, \cite[Section 4.1]{CDW2024}.} To the best of our knowledge, it looks like that there are no relevant literature to study the well-posedness of such Hamiltonian system. {However, if an equivalent Hamiltonian system can be identified whose well-posedness is easier to verify, then the well-posedness of \eqref{Hamiltonian} can be derived via a suitable equivalence transformation. To this end, we first construct a family of equivalent cost functionals, whose definition is given as follows.
\begin{definition}
    For a given controlled system, if there exist two cost functionals $J$ and $\bar J$ satisfying: for any admissible control $\tilde u$ and $\check u$, $J(\tilde u)<J(\check u)$ if and only if $\bar J(\tilde u)<\bar J(\check u)$, we say $J$ is equivalent to $\bar J$.
\end{definition}

Inspired by \cite{Y2013}, we adopt the so-called ``relaxed compensator" to construct an equivalent cost functional corresponding to \eqref{lcost}. We begin by introducing the following space
\begin{equation*}
  \Upsilon([0,T];\mathbb S^n)=\bigg\{P: [0,T]\rightarrow \mathbb S^n \ \Big|\ P_t=P_0+\int_0^t\dot{P}_s\mathrm{d}s,\ t\in [0,T] \bigg\}.
\end{equation*}
In linear-quadratic (LQ) control problems, quadratic terms play a dominant role. In fact, the cost functionals with linear terms can often be converted into purely quadratic forms through completing the square and suitable variable transformations.  As we all know, the value function of the SLQ problem can be represented by $\frac{1}{2}\langle \Pi_0x,x\rangle$, where $\Pi$ solves the Riccati equation \eqref{Riccati}, see \cite[Theorem 6.6.1]{Yong1999}. Motivated by the calculation in \cite[pp. 316-317]{Yong1999}, where they consider the difference between $J_i(u^i_{\cdot})$ and $\frac{1}{2}\langle \Pi_0 x,x\rangle$ to deal with the definite case. In contract, for the indefinite case, we consider the difference between $J_i(u^i_{\cdot})$ and  $\frac{1}{2}\langle P_0x,x\rangle$, where $P\in \Upsilon([0,T];\mathbb S^n)$.} Then we introduce the following notations:
\begin{equation}\label{notation}
  \left\{
  \begin{aligned}
    &Q^P=Q+\dot{P}+PA+A^{\top}P+D^{\top} PD,\quad S^P=S+PB+D^{\top}PF,\quad R^P=R+F^{\top}PF,\\
    &q^P=q+Pb+D^{\top}P\bar b+(P\bar A+D^{\top}P\bar D-\alpha_1 Q)x^0+(P\bar B+D^{\top}P\bar F-\beta_2S)u^0,\\
    &r^P=r+F^{\top}P\bar b+(F^{\top}P\bar D-\alpha_2S^{\top})x^0+(F^{\top}P\bar F-\beta_1R)u^0,\quad m^P_T=\langle \alpha_3^2L_Tx_T^0-2\alpha_3l_T,x_T^0\rangle,\\
    &M^P=\langle (\alpha_1^2Q+\bar D^{\top}P\bar D)x^0+2\bar D^{\top}P\bar b-2\alpha_1 q,x^0\rangle+2\alpha_2\beta_2\langle Su^0,x^0\rangle +\bar b^{\top}P\bar b+\sigma^{\top}P\sigma\\
    &  +\bar\sigma^{\top}P
    \bar\sigma+\langle (\beta_1^2R+\bar F^{\top}P\bar F)u^0+2\bar F^{\top}P\bar b-2\beta_1r, u^0\rangle,\quad  L^P_T= L_T-P_T,\quad  l^P_T=l_T-\alpha_3 L_Tx_T^0.
  \end{aligned}
  \right.
\end{equation}
According to above notations, we define
\begin{equation}\label{JiP}
  \begin{aligned}
    J_i^P(u_{\cdot}^i)
    &=\frac{1}{2}\mathbb{E}\bigg[\int_0^T\big\{\langle Q^P_tX_t^i+2q^P_t,X_t^i\rangle+ 2\langle S^P_tu_t^i,X_t^i\rangle+M^P_t\\
    &\qquad+\langle R^P_tu_t^i+2r^P_t,u_t^i\rangle\big\}\mathrm{d}t+\langle L^P_TX_T^i+2l^P_T,X_T^i\rangle+m^P_T\bigg].
  \end{aligned}
\end{equation}

Then we introduce the following auxiliary relaxed problem.

\textbf{Problem (MFP).} For $i$-th agent $\mathcal{A}_i, i=1,2,\cdots,N$, find $\bar u^{i,P} \in \mathscr{U}_{d}^{i}$ such that
\begin{equation*}
  J_i^P(\bar u_{\cdot}^{i,P})=\inf_{u_i\in \mathscr{U}_{d}^{i}}J_i^P(u_{\cdot}^i).
\end{equation*}

{Next, we introduce the definition of relaxed compensator,
\begin{definition}\label{def36}
  If there exists a $P \in \Upsilon([0,T];\mathbb S^n)$ such that   $(Q^P, S^P, R^P, L^P_T)$ satisfies
  \begin{equation*}
      \textbf{Condition (PD)} \qquad \left(\begin{matrix}
  Q &S \\
  S^{\top} &R
\end{matrix}\right)\geq 0,\qquad R \gg 0,\qquad L_T\geq 0,
  \end{equation*}
  then $P$ is called a relaxed compensator for Problem (MFD).
\end{definition}

The following lemma gives the relationship between $J_i(u_{\cdot}^i)$ and $J_i^P(u_{\cdot}^i)$, whose proof can be seen Appendix \ref{Appendix B}.
\begin{lemma}\label{Lem35}
  Suppose (H1)-(H2) hold and $P \in \Upsilon([0,T];\mathbb S^n)$. For any given $x\in \mathbb R^n$ and any admissible strategy $u^i\in \mathscr{U}_{d}^{i}$, we have
  \begin{equation}\label{0313111}
    J_i(u_{\cdot}^i)=J_i^P(u_{\cdot}^i)+\frac{1}{2}\langle P_0x,x\rangle.
  \end{equation}
  Moreover, if there exists a relaxed compensator $P \in \Upsilon([0,T];\mathbb S^n)$, Problem (MFD) is well-posed.
\end{lemma}

Then we give a useful result, which is so-called Schur's complement, \cite[Theorem 1]{A1969}.
\begin{lemma}[Schur's complement]\label{Schur}
  Let $Q\in \mathbb S^n$, $R\in \mathbb S^m$ and $S\in \mathbb R^{n\times m}$. Then the following statements are equivalent:

  i) $R\gg 0$ and $Q-SR^{-1}S^{\top}\geq 0$;\qquad
  ii) $R\gg 0$ and $\left(\begin{smallmatrix}
    Q&S\\
    S^{\top}&R
  \end{smallmatrix}\right)\geq 0$.
\end{lemma}

\begin{remark}\label{Rem31}
  By Schur's complement, it is obvious that if quadruple $(Q,S,R,L_T)$ satisfies Condition (PD), Problem (MFD) is well-posed.
   In fact, let ${\bf q}=q-\alpha_1Qx^0-\beta_2S u^0$, ${\bf r}=r-\beta_1Ru^0-\alpha_2S^{\top}x^0$, ${\bf l}_T=l_T-\alpha_3 L_Tx_T^0$, and $C_0=\mathbb E[\int_0^T \{\langle \alpha_1^2Q_tx_t^0-2\alpha_1q_t,x_t^0\rangle +\langle \beta_1^2R_tu_t^0-2\beta_1r_t,u_t^0\rangle +2\langle \alpha_2\beta_2 S_tu_t^0, x_t^0\rangle \}\mathrm{d}t\!+\!\langle \alpha_3^2L_Tx_T^0-2\alpha_3l_T,x_T^0\rangle] $,   then the cost {(16)} can be rewritten as
  \begin{equation*}
      \begin{aligned}
          J_i(u^i_{\cdot})&\!=\!\frac{1}{2}\mathbb E\bigg[\int_0^T\big\{\langle Q_tX_t^i\!+\!2{\bf q}_t, X_t^i\rangle \!+\!\langle R_tu_t^i\!+\!2{\bf r}_t, u_t^i\rangle \!+\!2\langle S_tu_t^i, X_t^i\rangle  \big\}\mathrm{d}t\!+\!\langle L_TX_T^i\!+\!2{\bf l}_T, X_T^i\rangle  \bigg]\!+\!C_0.
      \end{aligned}
  \end{equation*}
  By \cite[Theorem 6.4.2]{Yong1999} (see also \cite[Proposition 2.5.1]{SY2020} for the inhomogeneous case) and Schur's complement, if quadruple $(Q,S,R,L_T)$ satisfies Condition (PD), Problem (MFD) is well-posed.
\end{remark}

Now let us give a necessary and sufficient condition checking the relaxed compensator for Problem (MFD).

 \textbf{Condition (RC)} $P$ satisfies the following system of inequalities
 \begin{equation}\label{Riccaticondition}
   \left\{
   \begin{aligned}
     &\dot{P}_t+P_tA_t+A_t^{\top}P_t+D_t^{\top}P_tD_t+Q_t\\
     &\qquad-(S_t+P_tB_t+D_t^{\top}P_tF_t)(R_t+F_t^{\top}P_tF_t)^{-1} (S_t+P_tB_t+D_t^{\top}P_tF_t)^{\top}\geq 0,\\
     &P_T\leq L_T,\quad R_t+F_t^{\top}P_tF_t\gg 0,\quad t\in [0,T].
   \end{aligned}
   \right.
 \end{equation}
 \begin{proposition}\label{prop310}
   A function $P\in \Upsilon([0,T];\mathbb S^n)$ is a relaxed compensator for Problem (MFD) if and only if $P$ satisfies Condition (RC).
 \end{proposition}
 \begin{proof}
   Noticing Definition \ref{def36}, we know that $P$ is a relaxed compensator if and only if the quadruple $(Q^P,S^P,R^P,L^P_T)$ satisfies Condition (PD), i.e.,
   \begin{equation}\label{8128}
     \begin{aligned}
       R^P\gg 0,\qquad \left(\begin{matrix}
         Q^{P}_t&S^P_t\\
         (S^P_t)^{\top}& R^P_t
       \end{matrix}\right)\geq 0,\qquad L^P_T\geq 0.
     \end{aligned}
   \end{equation}
   Lemma \ref{Schur} yields that \eqref{8128} is equivalent to
   \begin{equation}\label{8129}
     \begin{aligned}
       R^P\gg 0, \qquad Q^P_t-S^P_t(R^P_t)^{-1}(S^P_t)^{\top}\geq 0,\qquad L_T^P\geq 0.
     \end{aligned}
   \end{equation}
   Recalling \eqref{notation}, we know that \eqref{8129} is equivalent to \eqref{Riccaticondition}.
 \end{proof}}

 Next, we will show the unique solvability of Hamiltonian system \eqref{Hamiltonian}. To do this, we introduce the corresponding Hamiltonian system of Problem (MFP),
 \begin{equation}\label{pHamiltonian}
  \left\{
    \begin{aligned}
      &\mathrm{d}\bar X^{i,P}_t=\big(A_t\bar X^{i,P}_t+B_t\bar u^{i,P}_t+\bar A_tx_t^0+\bar B_tu_t^0+b_t\big) \mathrm{d}t +\sigma_t\mathrm{d}W_t^i\\
      &\qquad \quad+\big(D_t\bar X^{i,P}_t+F_t\bar u^{i,P}_t+\bar D_tx_t^0+\bar F_t u_t^0+\bar b_t\big)\mathrm{d}W_t^0+\bar \sigma_t\mathrm{d}\overline W_t^i,\quad \bar X^{i,P}_0=x,\\
      &\mathrm{d}\varphi^{i,P}_t=-\big(A^{\top}_t\varphi^{i,P}_t+D^{\top}_t\eta^{i,P}_t+Q^P_t\bar X^{i,P}_t+S^P_t\bar u^{i,P}_t+ q^P_t\big)\mathrm{d}t\\
      &\qquad\quad +\eta^{i,P}_t\mathrm{d}W_t^0 +\zeta^{i,P}_t\mathrm{d} W_t^i+\vartheta^{i,P}_t\mathrm{d}\overline W_t^i, \quad \varphi^{i,P}_T=L^P_T\bar X^{i,P}_T+l^P_T,\\
      &R^P_t\bar u^{i,P}_t+B^{\top}_t\hat\varphi^{i,P}_t +F^{\top}_t\hat\eta^{i,P}_t+ (S^P_t)^{\top}\hat{\bar X}^{i,P}_t +r^P_t=0,\\
    \end{aligned}
  \right.
\end{equation}

\begin{proposition}\label{0309311}
  If there exists a relaxed compensator $P\in \Upsilon([0,T];\mathbb S^n)$, then Hamiltonian system \eqref{Hamiltonian} admits a unique solution $(\bar X^i,\bar u^i, \varphi^i,\eta^i,\zeta^i,\vartheta^i) \in S^2_{\mathscr{F}^i}([0,T];\mathbb R^n)\times \mathscr{U}_{d}^{i}\times \bar{\mathcal M}_{\mathscr{F}^i}$. Moreover, $(\bar X^i,\bar u^i)$ is the unique optimal pair of Problem (MFD).
\end{proposition}
\begin{proof}
  We know that if $(\bar X^{i,P}, \bar u^{i,P}, \varphi^{i,P}, \eta^{i,P}, \zeta^{i,P},\vartheta^{i,P})$ solves \eqref{pHamiltonian}, then
  \begin{equation}\label{transform}
    \left\{
    \begin{aligned}
      &\bar X^i=\bar X^{i,P},\quad \bar u^i=\bar u^{i,P},\quad \varphi^i = \varphi^{i,P}+P\bar X^{i,P},\  &&\zeta^i=\zeta^{i,P}+P \sigma,\\
      &\eta^i=\eta^{i,P}+P(D\bar X^{i,P}+F\bar u^{i,P}+ \bar Dx^0+\bar Fu^0+\bar b),\  &&\vartheta^i=\vartheta^{i,P}+P\bar \sigma,
    \end{aligned}
    \right.
  \end{equation}
  is a solution to \eqref{Hamiltonian}. Thus the well-posedness of \eqref{Hamiltonian} can be obtained by the well-posedness of \eqref{pHamiltonian}. Due to the reversibility of the transformation \eqref{transform}, the well-posedness of \eqref{Hamiltonian} also implies that of \eqref{pHamiltonian}. Therefore, the solvability between \eqref{Hamiltonian} and \eqref{pHamiltonian} are equivalent.

  It is easy to check that the coefficients of \eqref{pHamiltonian} satisfy the monotonicity condition (see \cite{CDW2024}), then it follows that \eqref{pHamiltonian} admits a unique solution $(\bar X^{i,P},\bar u^{i,P}, \varphi^{i,P},\eta^{i,P}, \zeta^{i,P}, \vartheta^{i,P})\in S^2_{\mathscr{F}^i}(0,T;\mathbb R^n)\times \mathscr{U}_{d}^{i}\times \bar{\mathcal M}_{\mathscr{F}^i}$. Moreover, we know that $(\bar X^{i,P},\bar u^{i,P})$ is the unique optimal pair of Problem (MFP) by a similar argument of Proposition \ref{Thm34}. In virtue of the transformation \eqref{transform}, we obtain $(\bar X^i,\bar u^i)=(\bar X^{i,P},\bar u^{i,P})$. Noticing the equivalence between the cost functional $J_i(u_{\cdot}^i)$ and $J_i^P(u_{\cdot}^i)$, (see Lemma \ref{Lem35}), we obtain that the unique optimal pair $(\bar X^i, \bar u^i)=(\bar X^{i,P},\bar u^{i,P})$ of Problem (MFP) is also the unique optimal pair of Problem (MFD).
\end{proof}

\subsection{Consistency Condition}

In Proposition \ref{Thm34}, we derive the agent $\mathcal{A}_i$'s optimal decentralized strategy $\bar u^i$ through the Hamiltonian system \eqref{Hamiltonian}, which is still parameterized by the undetermined limit process $(x^0,u^0)$. Now, we try to determine them by virtue of the consistency condition (CC).

Noting that the control weighting matrix $R$ is indefinite, we cannot obtain the explicit form of the optimal decentralized strategy $\bar u^i$ by the stationary condition \eqref{optimalcontrol}. Instead, with transformation \eqref{transform}, we have $(\bar X^i, \bar u^i)\equiv(\bar X^{i,P},\bar u^{i,P})$. Thus in the following we use $(\bar X^{i,P},\bar u^{i,P})$ to obtain $(x^0,u^0)$. Note that, for any $i\ne j$, $\bar X^{i,P}$ and $\bar X^{j,P}$ are identically distributed and conditionally independent {(under $\mathbb E[\cdot|\mathscr{F}^{\theta}]$, noticing that $\mathscr F^{\theta}=\mathscr F^{W^0}$)}.  Thus by the conditional strong law of large number, we have (the convergence is in the sense of almost surely; see \cite{Majerek2005,Hu20182}),
\begin{equation}\label{px0}
  \begin{aligned}
  &x^0=\lim_{N\rightarrow \infty}\frac{1}{N}\sum_{i=1}^N\bar X^i=\lim_{N\rightarrow \infty}\frac{1}{N}\sum_{i=1}^N\bar X^{i,P}=\mathbb E[\bar X^{i,P}|\mathscr{F}^{\theta}],\\
  &u^0=\lim_{N\rightarrow \infty}\frac{1}{N}\sum_{i=1}^N \bar u^i=\lim_{N\rightarrow \infty}\frac{1}{N}\sum_{i=1}^N \bar u^{i,P}=\mathbb{E}[\bar u^{i,P}|\mathscr{F}^{\theta}].
  \end{aligned}
\end{equation}
Moreover, recalling notations \eqref{notation} and combining \eqref{pHamiltonian} with \eqref{px0}, we obtain
\begin{equation}\label{92133}
  \begin{aligned}
    \bar u^{i,P}_t&=-(R^P_t)^{-1}\big[B_t^{\top}\hat{\varphi}_t^{i,P}+F_t^{\top}\hat{\eta}_t^{i,P}
    +(S^P_t)^{\top}\hat{\bar X}_t^{i,P}+r_t^P\big],\\
    u_t^0&=-\bar{\mathcal R}_t^{-1}\big(B_t^{\top}\check{\varphi}_t^{i,P}+F_t^{\top}\check{\eta}_t^{i,P}
    +\bar S_t \check{\bar X}_t^{i,P}+F^{\top}_tP_t\bar b_t+r_t\big),
  \end{aligned}
\end{equation}
where $\bar{\mathcal R}=R^P+F^{\top}P\bar F-\beta_1R$, $\bar S=(S^P)^{\top}+F^{\top}P\bar D-\alpha_2S^{\top}$ and {$\check{\bar X}^{i,P}=x^0$}. Here, the symbols $\check{\bar X}^{i,P}, \check{\varphi}^{i,P}, \check{\eta}^{i,P}$ are defined by \eqref{filteringsymbol}. Furthermore, substituting the optimal strategy $\bar u_{\cdot}^{i,P}$ into \eqref{pHamiltonian} and noticing that all agents are statistically identical, i.e., we can suppress subscript ``i", we could obtain the following CC system arises for generic agent,
 \begin{equation}\label{pCC}
  \left\{
    \begin{aligned}
      &\mathrm{d}\bar X^P_t=\big\{A_t\bar X_t^P+\bar A_t\check{\bar X}_t^P-B_t(R_t^P)^{-1}[B_t^{\top} \hat\varphi_t^P +F_t^{\top}\hat\eta_t^P+ (S_t^P)^{\top} \hat{\bar X}_t^P+(F^{\top}_tP_t\bar D_t-\alpha_2S^{\top}_t)\check{\bar X}_t^P\\
      &\qquad +F_t^{\top}P_t\bar b_t+r_t] -\mathcal A_t(B_t^{\top}\check{\varphi}_t^P+F_t^{\top}\check{\eta}_t^P+\bar S_t\check{\bar X}_t^P+F^{\top}_tP_t\bar b_t+r_t)+b_t\big\} \mathrm{d}t+\big\{D_t\bar X_t^P+\bar D_t\check{\bar X}_t^P \\
      &\qquad -F_t(R_t^P)^{-1}[B_t^{\top} \hat\varphi_t^P +F_t^{\top}\hat\eta_t^P+ (S_t^P)^{\top} \hat{\bar X}_t^P+(F^{\top}_tP_t\bar D_t-\alpha_2S^{\top}_t)\check{\bar X}_t^P+F_t^{\top}P_t\bar b_t+r_t]\\
      &\qquad -\mathcal B_t(B_t^{\top}\check{\varphi}_t^P+F_t^{\top}\check{\eta}_t^P+\bar S_t\check{\bar X}_t^P+F^{\top}_tP_t\bar b_t+r_t)+\bar b_t\big\}\mathrm{d}W_t^0+\sigma_t\mathrm{d}W_t+\bar \sigma_t\mathrm{d}\overline W_t,\\
      &\mathrm{d}\varphi^P_t=-\big\{A_t^{\top}\varphi_t^P+D_t^{\top}\eta_t^P+Q_t^P\bar X_t^P-(S_t^P)^{\top}(R_t^P)^{-1}[B_t^{\top}\hat\varphi_t^P+F^{\top}_t\hat\eta_t^P+(S_t^P)^{\top} \hat{\bar X}_t^P+r_t\\
      &\qquad +F_t^{\top}P_t\bar b_t+(F_t^{\top}P_t\bar D_t-\alpha_2S_t^{\top})\check{\bar X}_t^P]-\mathcal C_t(B^{\top}_t\check{\varphi}_t^P+F^{\top}_t \check{\eta}_t^P+\bar S_t\check{\bar X}_t^P+r_t+F_t^{\top}P_t\bar b_t)+ q_t\\
      &\qquad +P_tb_t+D_t^{\top}P_t\bar b_t+(P_t\bar A_t+D_t^{\top}P_t\bar D_t-\alpha_1 Q_t)\check{\bar X}_t^P\big\}\mathrm{d}t +\eta^P_t\mathrm{d}W_t^0+\zeta_t^P\mathrm{d} W_t+\vartheta^P_t\mathrm{d}\overline W_t,\\
      &\bar X^P_0=x,\quad \varphi^P_T=L^P_T\bar X^P_T+l_T-\alpha_3 L_T\check{\bar X}^P_T,
    \end{aligned}
  \right.
\end{equation}
where $\mathcal A=[\bar B-B(R^P)^{-1}(F^{\top} P\bar F -\beta_1R)]\bar{\mathcal R}^{-1}$, $\mathcal B=[\bar F-F(R^P)^{-1}(F^{\top} P\bar F-\beta_1R) ]\bar{\mathcal R}^{-1}$, $\mathcal C=[P\bar B+D^{\top}P\bar F-\beta_2S-(S^P)^{\top}(R^P)^{-1}(F^{\top}P\bar F-\beta_1R)]\bar{\mathcal R}^{-1}$.

CC system \eqref{pCC} is a fully coupled FBSDE with two types of conditional expectations. For a given relaxed compensator $P\in \Upsilon([0,T];\mathbb S^n)$, the unique solvability of CC system \eqref{pCC} can be investigate by the discounting method, see \cite[Section 4.2]{CDW2024}. However, since the coefficients depend on the choice of $P$, it causes lots of efforts in formulating the assumptions required for applying the discounting method. Although the CC system \eqref{pCC} may still be well-posed over a small time interval, this is not practically useful. In Section \ref{sec:feedback}, we compute the limiting process $(x^0, u^0)$ via \eqref{x0412}.

\subsection{Feedback representation}\label{sec:feedback}

{In this subsection, we give the feedback representation of the optimal decentralized strategies of Problem (MFD). We first introduce two Riccati equations,
\begin{equation}\label{Riccati}
  \left\{
  \begin{aligned}
      &\dot{\Pi}_t+\Pi_t A_t+A_t^{\top}\Pi_t +D_t^{\top}\Pi_t D_t+Q_t-\widetilde\Pi_t\mathcal{R}_{t}^{-1}\widetilde \Pi_t^{\top}= 0,\\
      &\Pi_T=L_T,\quad \mathcal R\gg 0,\quad t\in [0,T],
  \end{aligned}
  \right.
\end{equation}
\begin{equation}\label{Riccati2}
\left\{
\begin{aligned}
  &\dot{\Sigma}_t+\Sigma_t(A_t+\bar A_t)+A_t^{\top}\Sigma_t+D^{\top}_t\Sigma_t(D_t+\bar D_t) -\widetilde\Sigma_t \widetilde{\mathcal R}^{-1}_t\bar\Sigma_t^{\top}
  +(1-\alpha_1)Q_t=0,\\
  &\Sigma_T=(1-\alpha_3) L_T,\quad \widetilde{\mathcal R}+\widetilde{\mathcal R}^{\top}\gg 0, \quad t\in [0,T],
\end{aligned}
\right.
\end{equation}
and ODE,
\begin{equation}\label{BSDE0}
  \begin{aligned}
      &\dot{\rho}_t+A^{\top}_t\rho_t-\widetilde\Sigma_t \widetilde{\mathcal R}_t^{-1}\tilde \rho_t+\Sigma_tb_t+ D^{\top}_t\Sigma_t\bar b_t+q_t=0,\quad
      \rho_T=l_T,
  \end{aligned}
\end{equation}
where $\mathcal{R}=R+F^{\top}\Pi F$, $\widetilde{\mathcal R}=(1-\beta_1)R+F^{\top}\Sigma (F+\bar F)$, $\widetilde\Sigma=\Sigma (B+\bar B)+D^{\top}\Sigma (F+\bar F)+(1-\beta_2)S$, $\bar\Sigma =\Sigma B+(D+\bar D)^{\top}\Sigma F+(1-\alpha_2) S$, $\widetilde\Pi= \Pi B+D^{\top}\Pi F+S$, $\tilde\rho=B^{\top}\rho+F^{\top}\Sigma\bar b+r$. Then we have the following result, whose proof can be seen Appendix \ref{Appendix D}.

\begin{theorem}\label{thm8314}
  Under (H1)-(H2), if there exists a relaxed compensator $P\in \Upsilon([0,T]; \mathbb S^n)$, Riccati equation \eqref{Riccati} admits a unique solution $\Pi\in C([0,T];\mathbb S^n)$. Moreover, suppose that Riccati equation \eqref{Riccati2} admits a unique solution $\Sigma\in C([0,T];\mathbb R^{n\times n})$. then  for any given $x\in \mathbb R^n$, the optimal decentralized strategy $\bar u^i$ of Problem (MFD) has the following feedback representation
  \begin{equation}\label{feed230129}
    \begin{aligned}
      \bar u_t^i=-\mathcal{R}_t^{-1}\widetilde\Pi_t^{\top}(\hat{\bar X}_t^i-x_t^0)-\widetilde{\mathcal R}_t^{-1}(\bar\Sigma_t^{\top} x_t^0+\tilde\rho_t).
    \end{aligned}
  \end{equation}
Furthermore, the state-average limiting process $x^0$ is the unique solution to the following SDE,
\begin{equation}\label{x0412}
    \begin{aligned}
      \mathrm{d}x_t^0&=\big\{[A_t+\bar A_t-(B_t+\bar B_t)\widetilde{\mathcal R}^{-1}_t\bar \Sigma_t^{\top}]x_t^0-(B_t+\bar B_t)\widetilde{\mathcal R}_t^{-1}\tilde\rho_t+b_t\big\}\mathrm{d}t\\
      &\quad+\big\{[D_t+\bar D_t-(F_t+\bar F_t)\widetilde{\mathcal R}^{-1}_t\bar \Sigma_t^{\top}]x_t^0-(F_t+\bar F_t)\widetilde{\mathcal R}_t^{-1}\tilde\rho_t+\bar b_t\big\}\mathrm{d}W_t^0,\quad
      x_0^0=x,
    \end{aligned}
\end{equation}
and the optimal state ${\bar X}^i$ solves the following SDE,
\begin{equation}\label{feedback x}
  \begin{aligned}
    &\mathrm{d}{\bar X}_t^i=\big\{A_t\bar X_t^i-B_t\mathcal R_t^{-1}\widetilde\Pi_t^{\top}\hat{\bar X}_t^i+ \big[\bar A_t+B_t\mathcal R_t^{-1}\widetilde\Pi_t^{\top}-(B_t+\bar B_t)\widetilde{\mathcal R}_t^{-1}\bar\Sigma_t^{\top}\big]x_t^0+b_t\\
    &\qquad   -(B_t+\bar B_t)\widetilde{\mathcal R}_t^{-1}\tilde\rho_t\big\}\mathrm{d}t+\big\{D_t\bar X_t^i-F_t\mathcal R_t^{-1}\widetilde\Pi_t^{\top}\hat{\bar X}_t^i+ \big[\bar F_t+F_t\mathcal R_t^{-1}\widetilde\Pi_t^{\top}\\
    &\qquad- (F_t+\bar F_t)\widetilde{\mathcal R}_t^{-1}\bar\Sigma_t^{\top}\big]x_t^0-(F_t+\bar F_t)\widetilde{\mathcal R}_t^{-1}\tilde\rho_t +\bar b_t\big\}\mathrm{d}W_t^0 +\sigma_t\mathrm{d}W_t^i+\bar\sigma_t\mathrm{d}\overline W_t^i,\quad
    {\bar X}_0^i=x.
  \end{aligned}
\end{equation}
\end{theorem}

\begin{remark}\label{remark222}

  (i) If there exists $P\in \Upsilon([0,T];\mathbb S^n)$, then the transformation \eqref{trans} yields that $P\leq \Pi$.

  (ii)  Noticing that Riccati equation \eqref{Riccati2} is an asymmetric indefinite equation with complex structure, its solvability is still an open problem. To discuss the well-posedness of \eqref{Riccati2}, we further assume that $\bar A=\delta E_n$, $\bar B=\bar D=\bar F=0$, $\alpha_2=\beta_2$ with constant $\delta$ and $n$-dimensional identity matrix $E_n$. Then \eqref{Riccati2} becomes
  \begin{equation}\label{Riccati3}
\left\{
\begin{aligned}
  &\dot{\Sigma}_t+\Sigma_t(A_t+\bar A_t)+A_t^{\top}\Sigma_t+D^{\top}_t\Sigma_tD_t -\breve\Sigma_t \breve{\mathcal R}^{-1}_t\breve\Sigma_t^{\top}+(1-\alpha_1)Q_t=0,\\
  &\Sigma_T=(1-\alpha_3) L_T,\quad \widetilde{\mathcal R}\gg 0,\quad t\in [0,T],
\end{aligned}
\right.
\end{equation}
where $\breve{\mathcal R}= (1-\beta_1)R+F^{\top}\Sigma F$, $\breve\Sigma=\Sigma B+D^{\top}\Sigma F+(1-\beta_2)S$. Noticing that Riccati equation \eqref{Riccati3} is a symmetric indefinite equation, which solvability can be obtained by the first step in Appendix \ref{Appendix D}.

\end{remark}

There is an equivalence relationship between different conditions frequently used in the solution of indefinite LQ control problems. The proof can be found in Appendix \ref{Appendix C}.
\begin{lemma}\label{rmk031337}
    The following conditions are equivalent:

    (i) There exists a relaxed compensator for Problem (MFD).

    (ii) Riccati equation \eqref{Riccati} admits a unique solution $\Pi\in C([0,T];\mathbb S^n)$.

    (iii) The map $u^i\rightarrow J_i(u^i_{\cdot})$ is uniformly convex.
\end{lemma}

 \begin{remark}\label{rmk0313372}
  Based on the above analysis, one can find that relaxed compensator method is a more practical method compared to existing methods (Riccati equation or uniformly convex condition) for solving the indefinite problem. In fact, on the one hand, compared to solving directly the indefinite Riccati equation (see e.g. \cite{CLZ1998,CZ2000,RMZ2002,Xu2020}), it is easier to find a solution of inequality \eqref{Riccaticondition}, i.e., it is easier to find a relaxed compensator.

  On the other hand, compared to proposing the uniformly convex condition (i.e. the map $u\rightarrow J(u_{\cdot})$ is uniformly convex, see \cite{SLY2016,Wang20202,SXY2021}), the existence of relaxed compensator is easier to verify. This provides a more tractable, coefficient-based condition to guarantee uniform convexity.
\end{remark}}

\section{$\varepsilon$-Nash Equilibrium}\label{sec:varepsilon}

In the previous section, we obtain the  decentralized strategies $\bar{u}=(\bar{u}^1, \bar{u}^2, \cdots, \bar{u}^N)$, by introducing the auxiliary control problem and consistency condition, where (recall \eqref{92133}) $\bar u^i_{t}=\bar u_t^{i,P}=-(R^P_t)^{-1} [B_t^{\top}\hat{\varphi}_t^{i,P}+F_t^{\top}\hat{\eta}_t^{i,P}
+(S^P_t)^{\top}\hat{\bar X}_t^{i,P}+r_t^P]$. Next, we will verify that $\bar{u}^i$ is indeed an $\varepsilon$-Nash equilibrium.  {All results in this section are established under the same set of assumptions as those stated in Theorem  \ref{thm8314}, and these assumptions will not be restated prior to each lemma.} To do this, we first give the definition of $\varepsilon$-Nash equilibrium as follows.
\begin{definition}\label{epsilonnash}
  The control strategy $\bar{u}=(\bar{u}^1,\bar{u}^2,\cdots,\bar{u}^N)$, where $\bar{u}^i\in \mathcal{U}_{c}^{i}$, $1\leq i\leq N$, is called an $\varepsilon$-Nash equilibrium with respect to the cost functional $\mathcal{J}_i$, $1\leq i\leq N$, if there exists an $\varepsilon>0$, such that
  \begin{equation*}
    \mathcal{J}_i(\bar{u}_{\cdot}^i,\bar{u}_{\cdot}^{-i})\leq \mathcal{J}_i(u_{\cdot}^i,\bar{u}_{\cdot}^{-i})+\varepsilon,
  \end{equation*}
  where $u^i\in \mathcal{U}_{c}^{i}$ is any alternative control strategy for agent $\mathcal{A}_i$.
\end{definition}

{We assume that $\bar x^i$ is the corresponding state given by SDE \eqref{xx} with respect to $\bar u^i$ in the $N$ player model. Let $\bar x^{(N)}=\frac{1}{N}\sum_{i=1}^N\bar x^i$ be the average term, and $C_0$ is a constant independent of $N$, which may vary line by line. Then the following estimate holds.}
\begin{lemma}\label{lem72}
  \begin{equation*}
    \mathbb E\bigg[\sup_{0\leq t\leq T}\big|\bar x^{(N)}_t-x_t^0\big|^2\bigg]=O\Big(\frac{1}{N}\Big).
  \end{equation*}
\end{lemma}
\begin{proof}
  Recalling \eqref{xx} and \eqref{x0}, it holds that
  \begin{equation*}
    \left\{
    \begin{aligned}
      &\mathrm{d}(\bar x^{(N)}_t-x_t^0)=\big\{(A_t+\bar A_t)(\bar x^{(N)}_t-x_t^0)+(B_t+\bar B_t) \big(\bar u^{(N)}_t-u^0_t\big)\big\}\mathrm{d}t+ \frac{1}{N}\sum_{i=1}^N\sigma_t\mathrm{d}W_t^i\\
      &\qquad\qquad+\big\{(D_t+\bar D_t)(\bar x^{(N)}_t-x_t^0)+(F_t+\bar F_t) \big(\bar u^{(N)}_t- u^0_t\big)\big\}\mathrm{d}W_t^0 +\frac{1}{N}\sum_{i=1}^N\bar\sigma_t\mathrm{d}\overline W_t^i,\\
      &\bar x^{(N)}_0-x_0^0=0,
    \end{aligned}
    \right.
  \end{equation*}
  where $\bar u^{(N)}=\frac{1}{N}\sum_{i=1}^N\bar u^i$. Recalling \eqref{92133} and noticing \eqref{pHamiltonian}, we obtain $\bar u^i$ and $\bar u^j$ are identically distributed and conditional independent under $\mathbb{E}[\cdot|\mathscr{F}^{\theta}]$, then similar to \cite[Lemma 5.2]{Hu20182}, we have
  \begin{equation}\label{uN}
    \begin{aligned}
      \mathbb{E}\int_0^T\big|\bar u^{(N)}_t-\mathbb{E}[\bar u_t^i|\mathscr{F}^{\theta}_t]\big|^2\mathrm{d}t  = \frac{1}{N^2}\sum_{i=1}^N\mathbb{E} \int_0^T\big|\bar u_t^i-\mathbb{E}[\bar u_t^i|\mathscr{F}^{\theta}_t] \big|^2\mathrm{d}t \leq \frac{C_0}{N}=O\Big(\frac{1}{N}\Big).
    \end{aligned}
  \end{equation}
  Then, by Burkholder-Davis-Gundy (B-D-G) inequality  and Gronwall's inequality, we can complete the proof.
\end{proof}

Furthermore, recalling $\bar X^i_{\cdot}$ is the solution of \eqref{Hamiltonian}, we could obtain the following estimate.
\begin{lemma}\label{lem73}
  \begin{equation*}
    \sup_{1\leq i\leq N}\mathbb{E}\bigg[\sup_{0\leq t\leq T}|\bar x_t^i-\bar X_t^i|^2\bigg]=O\Big(\frac{1}{N} \Big).
  \end{equation*}
\end{lemma}
{\begin{proof}
    According to \eqref{xx} and \eqref{Hamiltonian}, it follows that
    \begin{equation*}
        \left\{
        \begin{aligned}
            &\mathrm{d}(\bar x_t^i-\bar X_t^i)=[A_t(\bar x_t^i-\bar X_t^i)+\bar A_t(\bar x^{(N)}_t-x^0_t)+\bar B_t(\bar u^{(N)}_t-u_t^0)]\mathrm{d}t\\
            &\qquad\quad +[D_t(\bar x_t^i-\bar X_t^i)+\bar D(\bar x^{(N)}_t-x^0_t)+\bar F_t(\bar u^{(N)}_t-u_t^0)]\mathrm{d}W_t^0,\\
            &\bar x_0^i-\bar X_0^i=0.
        \end{aligned}
        \right.
    \end{equation*}
    By Lemma \ref{lem72} and estimate \eqref{uN}, we get
    \begin{equation*}
        \mathbb E\bigg[\sup_{0\leq t\leq T}|\bar x_t^{(N)}-x^0_t|^2\bigg]\leq \frac{C_0}{N},\quad \text{and}\quad \mathbb E\bigg[\sup_{0\leq t\leq T}|\bar u_t^{(N)}-u^0_t|^2\bigg]\leq \frac{C_0}{N}.
    \end{equation*}
    By B-D-G inequality, we have
    \begin{equation*}
        \mathbb E\bigg[\sup_{0\leq t\leq T}|\bar x_t^i-\bar X_t^i|^2\bigg]\leq \frac{C_0}{N}+\mathbb E\int_0^T|\bar x_t^i-\bar X_t^i|^2\mathrm{d}t,
    \end{equation*}
    then the desired result can be obtained by Gronwall's inequality.
\end{proof}}

\begin{lemma}\label{lem112674}
  \begin{equation*}
    |\mathcal{J}_i(\bar u_{\cdot}^i,\bar u_{\cdot}^{-i})-J_i(\bar u_{\cdot}^i)|=O\Big(\frac{1}{\sqrt N}\Big),\qquad 1\leq i\leq N.
  \end{equation*}
\end{lemma}
\begin{proof}
  According to \eqref{cost} and \eqref{lcost}, we have
  \begin{equation*}
    \begin{aligned}
      &\mathcal{J}_i(\bar u_{\cdot}^i,\bar u_{\cdot}^{-i})-J_i(\bar u_{\cdot}^i)
      =\frac{1}{2}\mathbb{E}\Big[\int_0^T\big\{\langle Q_t(\bar x_t^i-\alpha_1\bar x^{(N)}_t)+ 2q_t,\bar x_t^i-\alpha_1\bar x^{(N)}_t\rangle-\langle Q_t(\bar X_t^i -\alpha_1 x_t^0)\\
      &\quad+2q_t,\bar X_t^i-\alpha_1 x_t^0\rangle+\langle R_t(\bar u_t^i-\beta_1 \bar u^{(N)}_t)+r_t, \bar u_t^i-\beta_1 \bar u^{(N)}_t\rangle-\langle R_t(\bar u_t^i-\beta_1 \bar u^{0}_t)+r_t, \bar u_t^i-\beta_1 \bar u^{0}_t\rangle \\
      &\quad+2\langle S_t(\bar u_t^i-\beta_2\bar u_t^{(N)}), \bar x_t^i-\gamma_2\bar x_t^{(N)}\rangle -2\langle S_t(\bar u_t^i-\beta_2\bar u_t^{0}), \bar X_t^i-\gamma_2 x_t^{0}\rangle\big\}\mathrm{d}t\\
      &\quad+\langle L_T(\bar x_T^i-\alpha_3\bar x^{(N)}_T)+2l_T,\bar x_T^i-\alpha_3\bar x^{(N)}_T\rangle-\langle L_T(\bar X_T^i-\alpha_3 x_T^0)+ 2l_T, \bar X_T^i-\alpha_3 x_T^0\rangle\Big].
    \end{aligned}
  \end{equation*}
  For the first part, noticing $\langle Qx, x\rangle-\langle Qy, y\rangle =\langle Q(x-y), x-y\rangle +2\langle Q(x-y), y\rangle$, we have
  {\begin{equation*}\label{112443}
    \begin{aligned}
      &\mathbb{E} \int_0^T\big\{\langle Q_t(\bar x_t^i-\alpha_1 \bar x^{(N)}_t), \bar x_t^i-\alpha_1 \bar x^{(N)}_t\rangle-\langle Q_t(\bar X_t^i-\alpha_1 x_t^0), \bar X_t^i-\alpha_1 x_t^0\rangle \big\}\mathrm{d}t \\
      &\leq \mathbb E\int_0^T\langle Q[\bar x_t^i-\bar X_t^i-\alpha_1(\bar x_t^{(N)}-x_t^0)], \bar x_t^i-\bar X_t^i-\alpha_1(\bar x_t^{(N)}-x_t^0)\rangle \mathrm{d}t\\
      &\qquad \quad +2\mathbb E\int_0^T\langle Q[\bar x_t^i-\bar X_t^i-\alpha_1(\bar x_t^{(N)}-x_t^0)], \bar X_t^i-\alpha_1 x_t^0\rangle \mathrm{d}t\\
      &\leq 2\mathbb E\int_0^T\big\{\langle Q(\bar x_t^i-\bar X_t^i), \bar x_t^i-\bar X_t^i\rangle+\alpha_1^2\langle Q(\bar x^{(N)}_t-x_t^0),\bar x^{(N)}_t-x_t^0\rangle   \big\}\mathrm{d}t\\
      &\qquad\quad +C_0\int_0^T\big[\mathbb{E}|\bar x_t^i-\alpha_1\bar x^{(N)}_t-(\bar X_t^i-\alpha_1x_t^0)|^2\big]^{\frac{1}{2}}\big[ \mathbb{E}|\bar X_t^i-\alpha_1x_t^0|^2\big]^{\frac{1}{2}}\mathrm{d}t\\
      &\leq C_0\int_0^T\Big\{\mathbb{E}|\bar x_t^i-\bar X_t^i|^2+\mathbb{E}|\bar x^{(N)}_t-x_t^0|^2 +\big[\mathbb{E}|\bar x_t^i-\bar X_t^i-\alpha_1(\bar x^{(N)}_t-x_t^0)|^2\big]^{\frac{1}{2}}\big[ \mathbb{E}|\bar X_t^i|^2+\mathbb E|x_t^0|^2\big]^{\frac{1}{2}}\Big\}\mathrm{d}t\\
      &\leq C_0\int_0^T\Big\{\mathbb{E}|\bar x_t^i-\bar X_t^i|^2+\mathbb{E}|\bar x^{(N)}_t-x_t^0|^2 +\big[ \mathbb{E}|\bar x_t^i-\bar X_t^i|^2+\mathbb{E}|\bar x^{(N)}_t-x_t^0|^2\big]^{\frac{1}{2}} \Big\}\mathrm{d}t=O\Big(\frac{1}{\sqrt{N}}\Big),
    \end{aligned}
  \end{equation*}}\noindent
  where the last inequality is due to Lemma \ref{lem72}, Lemma \ref{lem73} and $\mathbb{E}[\sup_{0\leq t\leq T}(| \bar X_t^i|^2+|x_t^0|^2)]\leq C_0$. Similarly, we also know that the second, third and fourth parts are all $\frac{1}{\sqrt N}$ order. Then we can obtain the desired result.
\end{proof}

Now we consider the perturbation to $i$-th agent, i.e. the agent $\mathcal{A}_i$ choose an alternative strategy $u^i\in \mathscr{U}_{c}^{i}$, while other agents $\mathcal{A}_j, j\ne i$ still take the decentralized strategy $\bar u^j$. Then the perturbed centralized state of $\mathcal{A}_k$, $k=1,2,\cdots,N$ is given by
\begin{equation}\label{pixx}
  \left\{
    \begin{aligned}
      \mathrm{d}x_t^i&=\big(A_tx_t^i+B_tu_t^i+\bar A_t x^{(N)}_t+\bar B_t u_t^{(N)}+b_t\big) \mathrm{d}t+ \sigma_t \mathrm{d}W_t^i\\
      &\quad+\big(D_t x_t^i+F_t u_t^i+\bar D_t x^{(N)}_t+\bar F_t u_t^{(N)} +\bar b_t\big)\mathrm{d}W_t^0+\bar\sigma_t\mathrm{d}\overline W_t^i,\quad        x_0^i=x,\\
      \mathrm{d}x_t^j&=\big(A_tx_t^j+B_t\bar u_t^j+\bar A_tx^{(N)}+\bar B_t u_t^{(N)}+b_t\big)\mathrm{d}t+\sigma_t\mathrm{d}W_t^j\\
      &\quad +\big(D_tx_t^j+F_t\bar u_t^j+\bar D_tx^{(N)}_t+\bar F_tu_t^{(N)}+\bar b_t\big)\mathrm{d}W_t^0+\bar\sigma_t\mathrm{d}W_t^j,\quad x_0^j=x,\quad j\ne i,
    \end{aligned}
  \right.
\end{equation}
where  $ x^{(N)}=\frac{1}{N}\sum_{k=1}^N x^k$ and $u^{(N)}=\frac{1}{N}(\sum_{j\ne i}\bar u^j+u^i)$. The cost functional of $i$-th agent $\mathcal{A}_i$ is
\begin{equation*}
  \begin{aligned}
    \mathcal J_i(u_{\cdot}^i,\bar u_{\cdot}^{-i})
    &\!=\!\frac{1}{2}\mathbb{E}\bigg[\!\int_0^T\!\!\!\big\{\langle Q_t( x_t^i -\alpha_1 x^{(N)}_t)\!+\! 2q_t, x_t^i\!-\!\alpha_1  x^{(N)}_t\rangle+\langle R_t(u_t^i-\beta_1 u_t^{(N)})\!+\!2r_t, u_t^i\!-\!\beta_1  u_t^{(N)}\rangle \\
    & \quad +2\langle S_t(u_t^i-\beta_2  u_t^{(N)}),  x_t^i-\alpha_2 x_t^{(N)}\rangle\big\}\mathrm{d}t+\langle L_T( x_T^i-\alpha_3  x^{(N)}_T)+2l_T,  x_T^i-\alpha_3  x^{(N)}_T\rangle\bigg].
  \end{aligned}
\end{equation*}
In addition, the related decentralized states of all agents with perturbation satisfy
\begin{equation}\label{lixx}
  \left\{
    \begin{aligned}
      \mathrm{d} X_t^i&=\big(A_t  X_t^i+B_t  u_t^i+\bar A_tx_t^0+\bar B_tu_t^0+b_t\big) \mathrm{d}t +\sigma_t\mathrm{d}W_t^i\\
      &\quad +\big(D_t  X_t^i+F_t  u_t^i+\bar D_tx_t^0+\bar F_tu_t^0+\bar b_t\big)\mathrm{d}W_t^0+\bar \sigma_t\mathrm{d}\overline W_t^i,\quad
       X_0^i=x,\\
      \mathrm{d} X_t^j&=\big(A_t  X_t^j+B_t \bar u_t^j+\bar A_tx_t^0+\bar B_tu_t^0+b_t\big) \mathrm{d}t +\sigma_t\mathrm{d}W_t^j\\
      &\quad +\big(D_t  X_t^j+F_t  u_t^j+\bar D_tx_t^0+\bar F_tu_t^0+\bar b_t\big)\mathrm{d}W_t^0+\bar \sigma_t\mathrm{d}\overline W_t^j,\quad
       X_0^j=x,\quad j\ne i, \\
    \end{aligned}
  \right.
\end{equation}

In order to show that $\bar u=(\bar u^1,\bar u^2,\cdots,\bar u^N)$ is an $\varepsilon$-Nash equilibrium, we need to prove $\mathcal{J}_i(\bar u_{\cdot}^i,\bar u_{\cdot}^{-i}) \leq \mathcal{J}_i( u_{\cdot}^i,\bar u_{\cdot}^{-i})+\varepsilon$. Thus we only consider the alternative strategy $u^i\in \mathscr{U}_{c}^{i}$ such that $\mathcal{J}_i(u_{\cdot }^i,\bar u_{\cdot}^{-i})\leq \mathcal{J}_i(\bar u_{\cdot}^i,\bar u_{\cdot}^{-i})$. In existing literature on LP system, usually assumes that the coefficients of the cost functional satisfies Condition (PD), and
\begin{equation}\label{important}
  \mathbb{E} \int_0^T|u_t^i|^2 \mathrm{d}t \leq C_0,
\end{equation}
which is the key inequality to proving the $\varepsilon$-Nash equilibrium, can be obtained by simple calculation. {Specifically, the boundedness condition \eqref{important} is introduced to guarantee that the perturbed cost functionals $\mathcal J_i(u^i_{\cdot}, \bar u_{\cdot}^{-i})$ and $J_i(u^i_{\cdot})$ admit desirable approximation properties (see Lemma \ref{lemma6235565}). Interesting readers can refer to \cite{HMC2007,Hu20182,CDW2024} for the positive-definite case. }  It is worth mentioning that compared with some literature on the indefinite MFG (see \cite{Xu2020,Wang2020,Wang20202}), our results have essential difference. \cite{Xu2020} directly assumes that all admissible strategies are uniformly bounded (i.e. $\sup_{1\leq i\leq N}\sup_{0\leq t\leq T}\mathbb E|u_t^i|^2<C_0$), then the inequality $\mathbb{E}\int_0^T|u_t^i|^2 \mathrm{d}t\leq C_0$ is obvious. \cite{Wang2020} and \cite{Wang20202} studied respectively MFGs with indefinite state weight and indefinite control weight in the cost functional. They all assume that the MFG problem is uniformly convex (i.e. the map $u\rightarrow J(u_{\cdot})$ is uniformly convex), then the inequality $\mathbb{E}\int_0^T|u_t^i|^2 \mathrm{d}t \leq C_0$ can be directly derived from the uniform convexity. In summary, the above literature adds some additional assumptions to make the inequality hold, rather than directly address the problem itself. As a contrast, our model allows that the state weight and control weight are all indefinite, and the assumptions of uniform boundedness of admissible strategy set or uniformly convexity of cost functional is no longer required. Let us now study how to prove the inequality $\mathbb{E}\int_0^T|u_t^i|^2 \mathrm{d}t\leq C_0$ holds without additional assumption.

Inspired by the method of equivalent cost functional, we would like to find an equivalent cost functional. We first give some notations as follows,
\begin{equation*}
    \left\{
      \begin{aligned}
        &\tilde{Q}^P=\alpha^2_1Q+\bar D^{\top}P\bar D,\quad  \bar{Q}^P =-\alpha_1 Q+\bar A^{\top}P+\bar D^{\top}P D, \quad  \tilde{q}^P=q+Pb+D^{\top}P\bar b,\\ &\bar{q}^P=-\alpha_1 q+\bar b^{\top}P\bar D,\quad \hat Q^P=-\beta_2 S+\bar B^{\top}P+\bar F^{\top}PD,\quad \tilde S^P=\bar F^{\top}P\bar D+\alpha_2\beta_2 S,\\
        &\tilde R^P=\beta_1 R+\bar F^{\top}P\bar F,\quad
        \tilde{r}^P=r+F^{\top}P\bar b,\quad \bar{r}^P=-\alpha_2S+\bar D^{\top}PF,\\
        &\hat q^{P}=\bar F^{\top}P\bar b-\beta_1 r,\quad \hat r^P=\bar F^{\top}PF-\beta_1 R,\quad \tilde{M}^P=\bar b^{\top}P\bar b+\sigma^{\top} P\sigma+\bar\sigma^{\top} P\bar\sigma.
      \end{aligned}
    \right.
\end{equation*}
We also introduce
\begin{equation*}
  \begin{aligned}
    &\mathcal{J}^P_i(u_{\cdot}^i,\bar{u}_{\cdot}^{-i})
    =\frac{1}{2}\mathbb{E}\Big[\int_0^T\big\{\langle Q^P_tx_t^i+2\tilde q^P_t+2S^P_tu^i_t,x_t^i\rangle +\langle R^P_t u_t^i+2\tilde r^P_t, u_t^i\rangle+\langle \tilde{Q}^P_t x^{(N)}_t, x^{(N)}_t\rangle\\
    &\qquad +2\langle \bar{Q}^P_tx_t^i+\bar{r}^P_tu_t^i+ \bar{q}^P_t, x^{(N)}_t\rangle+\langle \tilde R^P_{t}u_t^{(N)},u_t^{(N)}\rangle+2\langle\hat Q^P_{t}x_t^i+\hat r^P_tu_t^i+\hat q_t^P,u_t^{(N)}\rangle+\tilde{M}^P_t \\
    &\qquad+2\langle \tilde S^P_{t}u_t^{(N)},x_t^{(N)}\rangle\big\}\mathrm{d}t+\langle L^P_Tx_T^i+2l_T, x_T^i\rangle -2\alpha_3\langle L_Tx_T^i +l_T, x^{(N)}_T\rangle+\alpha_3^2\langle L_Tx^{(N)}_T,x^{(N)}_T\rangle  \Big].
  \end{aligned}
\end{equation*}
{where $(Q^P,S^P, R^P,L^P_T)$ is defined in \eqref{notation}. Compared with the auxiliary limiting cost $J_i^P$ (defined in \eqref{JiP} ) and notations \eqref{notation}, we have separated the limiting terms $(x^0,u^0)$ from the inhomogeneous terms and rewritten them as the average terms $(x^{(N)},u^{(N)})$.} Then we can show the following relationship between $\mathcal{J}_i(u_{\cdot}^i,\bar{u}_{\cdot}^{-i})$ and $\mathcal{J}^P_i(u_{\cdot}^i,\bar{u}_{\cdot}^{-i})$, which plays a key role in our analysis. The proof is similar to Lemma \ref{Lem35}, so we omit it.

\begin{lemma}\label{lemma6235555}
   For any given initial value $x\in \mathbb R^n$ and $u^i\in \mathcal{U}_{c}^{i}$, we have
  \begin{equation*}
    \mathcal{J}_i(u_{\cdot}^i,\bar u_{\cdot}^{-i})=\mathcal{J}^P_i(u_{\cdot}^i,\bar u_{\cdot}^{-i})+\frac{1}{2}\langle P_0x, x\rangle.
  \end{equation*}
\end{lemma}
\begin{remark}\label{remark6235566}
  For any $u^{i,1}, u^{i,2}\in \mathcal U_c^i$, we have

  (i) $\mathcal{J}_i(u_{\cdot}^{i,1},\bar{u}_{\cdot}^{-i})<\mathcal{J}_i(u_{\cdot}^{i,2},\bar{u}_{\cdot}^{-i})$ if and only if $\mathcal{J}^P_i(u_{\cdot}^{i,1},\bar{u}_{\cdot}^{-i})<\mathcal{J}^P_i(u_{\cdot}^{i,2},\bar{u}_{\cdot}^{-i})$;

  (ii) $\mathcal{J}_i(u_{\cdot}^{i,1},\bar{u}_{\cdot}^{-i})=\mathcal{J}_i(u_{\cdot}^{i,2},\bar{u}_{\cdot}^{-i})$ if and only if $\mathcal{J}^P_i(u_{\cdot}^{i,1},\bar{u}_{\cdot}^{-i})=\mathcal{J}^P_i(u_{\cdot}^{i,2},\bar{u}_{\cdot}^{-i})$;

  (iii) $\mathcal{J}_i(u_{\cdot}^{i,1},\bar{u}_{\cdot}^{-i})>\mathcal{J}_i(u_{\cdot}^{i,2},\bar{u}_{\cdot}^{-i})$ if and only if $\mathcal{J}^P_i(u_{\cdot}^{i,1},\bar{u}_{\cdot}^{-i})>\mathcal{J}^P_i(u_{\cdot}^{i,2},\bar{u}_{\cdot}^{-i})$.
\end{remark}
\begin{remark}
Motivated by Lemma \ref{lemma6235555} and Remark \ref{remark6235566}, to obtain the uniform estimate \eqref{important} under indefinite coefficients, it seems that we can introduce relaxed compensator, and then consider the alternative control $u^i\in \mathcal{U}_{c}^{i}$ such that $\mathcal{J}^P_i(u_{\cdot}^i,\bar{u}_{\cdot}^{-i})\leq \mathcal{J}^P_i(\bar{u}_{\cdot}^i,\bar{u}_{\cdot}^{-i})$. However, since the form of $\mathcal J^P_i(u^i_{\cdot},\bar u^{-i}_{\cdot})$ is very complex and cannot be rewritten as a completely square form, we cannot directly obtain $R^P_t\mathbb E\int_0^T|u_t^i|^2\mathrm{d}t\leq \mathcal J^P_i(u^i_{\cdot},\bar u_{\cdot}^{-i})$, which is an important step in proving  $\mathbb{E}\int_0^T|u_t^i|^2 \mathrm{d}t\leq C_0$.
\end{remark}

Next, we use the approximated method to obtain \eqref{important} as $N\rightarrow \infty$. To do this, we first present some estimates of the perturbed state and cost functional.
\begin{lemma}\label{lem81647}
  \begin{equation}\label{81653}\begin{aligned}
    \mathbb E\bigg[\sup_{0\leq t\leq T}\big| x^{(N)}_t-x_t^0\big|^2\bigg]&\leq \frac{C_0}{N}+\frac{C_0}{N^2}\mathbb E\int_0^T|u_t^i|^2\mathrm{d}t,\\
    \sup_{1\leq i\leq N}\mathbb{E}\bigg[\sup_{0\leq t\leq T}| x_t^i- X_t^i|^2\bigg]&\leq \frac{C_0}{N}+\frac{C_0}{N^2}\mathbb E\int_0^T|u_t^i|^2\mathrm{d}t.
  \end{aligned}\end{equation}
\end{lemma}
\begin{proof}
  Let $\delta x^{(N)}= x^{(N)}-x^0$, recalling \eqref{pixx} and \eqref{x0}, it holds that
  \begin{equation*}
    \left\{
    \begin{aligned}
      &\mathrm{d}\delta x^{(N)}_t=\bigg[(A_t+\bar A_t)\delta x^{(N)}_t+(B_t+\bar B_t) (\bar u_t^{(N)}-u^0_t)+\frac{B_t+\bar B_t}{N}( u_t^i-\bar u_t^i)\bigg]\mathrm{d}t+ \frac{1}{N}\sum_{i=1}^N\sigma_t\mathrm{d}W_t^i\\
      &\qquad\quad +\bigg[(D_t+\bar D_t)\delta x^{(N)}_t+(F_t+\bar F_t) (\bar u_t^{(N)}-u^0_t)+\frac{F_t+\bar F_t}{N}( u_t^i-\bar u_t^i)\bigg]\mathrm{d}W_t^0 +\frac{1}{N}\sum_{i=1}^N\bar\sigma_t\mathrm{d}\overline W_t^i,\\
      &\delta x^{(N)}_0=0.
    \end{aligned}
    \right.
  \end{equation*}
  By similar arguments as Lemma \ref{lem72}, we obtain the first inequality. Similarly, we can also obtain the second inequality.
\end{proof}

Similar to Lemma \ref{lem112674}, we also have the following lemma.
\begin{lemma}\label{lemma6235577}
  \begin{equation*}\begin{aligned}
    &\big|\mathcal{J}^P_i( {u}_{\cdot}^i,\bar{u}_{\cdot}^{-i})-J^P_i( {u}_{\cdot}^i)\big|\leq \frac{C_0}{\sqrt N}+\frac{C_0}{N}\mathbb E\int_0^T|u_t^i|^2\mathrm{d}t.
  \end{aligned}\end{equation*}
\end{lemma}
Using Lemma \ref{lemma6235577}, we obtain
\begin{equation*}
  -\frac{C_0}{\sqrt N}-\frac{C_0}{N}\mathbb E\int_0^T|u_t^i|^2\mathrm{d}t+J^P_i( {u}_{\cdot}^i)\leq  \mathcal{J}^P_i( {u}_{\cdot}^i,\bar{u}_{\cdot}^{-i}).
\end{equation*}
According to Lemma \ref{rmk031337}, we know that the map $u^i\rightarrow J^P_{i}(u^i_{\cdot})$ is uniformly convex, which implies that $\lambda \mathbb E\int_0^T|u^i_t|^2\mathrm{d}t-C_0\leq {J}_i^P(u_{\cdot}^i )$ for some $\lambda >0$. Thus
\begin{equation}\label{unbounded}
  \Big(\lambda-\frac{C_0}{N}\Big)\mathbb E\int_0^T|u_t^i|^2\mathrm{d}t-C_0\leq  \mathcal{J}^P_i( {u}_{\cdot}^i,\bar{u}_{\cdot}^{-i})\leq\mathcal{J}_i^P(\bar{u}_{\cdot}^i,\bar{u}_{\cdot}^{-i})\leq J_i^P(\bar{u}_{\cdot}^i)+O\Big(\frac{1}{\sqrt N}\Big),
\end{equation}
then we obtain that for sufficiently large $N$ ($N>\frac{C_0}{\lambda}$), it holds \eqref{important}. Then we have the following estimates.
\begin{lemma}\label{lemma6235565}
  \begin{equation*}\begin{aligned}
    \mathbb{E}\bigg[\sup_{0\leq t\leq T}| x^{(N)}_t-x_t^0|^2\bigg]&=O\Big(\frac{1}{N}\Big),\qquad
    \sup_{1\leq i\leq N}\mathbb{E}\bigg[\sup_{0\leq t\leq T}| x_t^i- X_t^i|^2\bigg]=O\Big(\frac{1}{N}\Big),\\
    |\mathcal{J}_i(u_{\cdot}^i,\bar{u}_{\cdot}^{-i})-J_i(u_{\cdot}^i)|&=O\Big(\frac{1}{\sqrt N}\Big),\quad 1\leq i\leq N.
  \end{aligned}\end{equation*}
\end{lemma}

Based on above lemmas, now we give the following main result of this section.
\begin{theorem}\label{thm410}
  The set of decentralized strategies $\bar{u}=(\bar{u}^1,\bar{u}^2,\cdots, \bar{u}^N)$, where $\bar{u}^i$ is given by \eqref{optimalcontrol}, is an $\varepsilon=O(\frac{1}{\sqrt N})$-Nash equilibrium.
\end{theorem}
\begin{proof}
  According to Lemma \ref{lem112674} and Lemma \ref{lemma6235565}, we have that, for $1\leq i\leq N$,
  \begin{equation*}
    \begin{aligned}
      \mathcal{J}_i(\bar{u}_{\cdot}^i,\bar{u}_{\cdot}^{-i})=J_i(\bar{u}_{\cdot}^i)+O\Big(\frac{1}{\sqrt N}\Big)
      \leq J_i(u_{\cdot}^i)+O\Big(\frac{1}{\sqrt N}\Big)
      =\mathcal{J}_{i}(u_{\cdot}^i,\bar{u}_{\cdot}^{-i}) +O\Big(\frac{1}{\sqrt N}\Big).
    \end{aligned}
  \end{equation*}
  Therefore, the result holds with $\varepsilon=O(\frac{1}{\sqrt N})$.
\end{proof}

\section{Application}\label{sec:application}

 In this section, we try to apply our theoretical results to solve the Problem (EX) introduced in Section \ref{motivation}. Noticing that the weighting matrix of the control in cost functional \eqref{cost1} is $0$, thus Condition (PD) {does not hold}. According to Proposition \ref{Thm34}, we can obtain the Hamiltonian system,
\begin{equation}\label{3359}
  \left\{
  \begin{aligned}
    &\mathrm{d}\bar x_t^i=(r_t\bar x_t^i+B_t\bar u_t^i-b_t)\mathrm{d}t+\sigma_t\bar u_t^i\mathrm{d}W_t^0+c_t\mathrm{d}W_t^i+\bar c_t\mathrm{d}\overline W_t^i,\quad \bar x_0^i=x_0,\\
    &\mathrm{d}\varphi_t^i=-r_t\varphi_t^i\mathrm{d}t+\eta_t^i\mathrm{d}W_t^0+\zeta_t^i\mathrm{d}W_t^i+\vartheta_t^i\mathrm{d}\overline W_t^i,\quad \varphi_T^i=\gamma (\bar x_T^i-x_T^0)-\frac{1}{2},\\
    &B_t\hat{\varphi}_t^i+\sigma_t\hat\eta_t^i=0.
  \end{aligned}
  \right.
\end{equation}
In \eqref{3359}, the SDE and BSDE are coupled through the third equation, which makes it difficult for us to obtain the well-posedness of \eqref{3359}.
 Motivated by Proposition \ref{0309311}, in order to solve this indefinite problem, we need to introduce the relaxed compensator $P$. By virtue of Proposition \ref{prop310}, we know that the relaxed compensator $P$ satisfies the following inequality,
\begin{equation}\label{030860}
  \begin{aligned}
    &\dot{P}_t+2r_tP_t+\frac{B^2_t}{\sigma_t^2}P_t\geq 0,\quad
    P_T\leq \gamma,\quad \sigma^2_tP_t>0,\quad t\in [0,T].
  \end{aligned}
\end{equation}
It is easy to check that $\gamma \exp(\int_t^T(2r_s-\frac{B_s^2}{\sigma_s^2})\mathrm{d}s)$ is a solution of {Inequality} \eqref{030860}. In fact, this relaxed compensator is also the solution of the first equation in \eqref{92362}. Then we introduce the following auxiliary FBSDE
\begin{equation}\label{8257}
  \left\{
  \begin{aligned}
    &\mathrm{d}\bar x_t^{i,P}=(r_t\bar x_t^{i,P}+B_t\bar u_t^{i,P}-b_t)\mathrm{d}t+ \sigma_t\bar u_t^{i,P}\mathrm{d}W_t^0+c_t\mathrm{d}W_t^i +\bar c_t\mathrm{d}\overline W_t^i,\\
    &\mathrm{d}\varphi_t^{i,P}=-(r_t\varphi_t^{i,P}+Q^P_t\bar{x}_t^{i,P}
    +B_tP_t\bar u_t^{i,P}-b_tP_t)\mathrm{d}t+\eta_t^{i,P}\mathrm{d}W_t^0+\zeta_t^{i,P}\mathrm{d} W_t^i+\vartheta_t^{i,P}\mathrm{d}\overline W_t^i,\\
    &\sigma^2_tP_t\bar u_t^{i,P}+B_t\hat{\varphi}_t^{i,P}+\sigma_t\hat{\eta}_t^{i,P}+B_tP_t\hat{\bar x}_t^{i,P}=0,\\
    &\bar{x}_0^{i,P}=x_0,\quad \varphi_T^{i,P}=(L_T-P_T)\bar{x}_T^{i,P}-\gamma x_T^0-\frac{1}{2},
  \end{aligned}
  \right.
\end{equation}
where $Q^P=\dot{P}+2rP$. One can easily check that FBSDE \eqref{8257} satisfies the monotonicity condition, then it admits a unique solution, see \cite{CDW2024}. Thus FBSDE \eqref{3359} admits a unique solution  by Proposition \ref{0309311}. Next, in order to obtain the feedback representation of decentralized strategies, we introduce the following equations
\begin{equation}\label{92362}
  \left\{
  \begin{aligned}
    &\dot{\Pi}_t+2r_t\Pi_t-\frac{B^2_t}{\sigma^2_t}\Pi_t= 0,\quad
    \Pi_T= \gamma,\quad \sigma_t^2\Pi_t >0,\\
    &\dot{\Sigma}_t+2r_t\Sigma_t-\frac{B^2_t}{\sigma^2_t\Pi_t}\Sigma_t^2=0,\quad
        \Sigma_T=0,\\
        &\dot{\rho}_t+r_t\rho_t-\frac{B^2_t\Sigma_t\rho_t}{\sigma^2_t\Pi_t}-b_t\Sigma_t=0,\quad
    \rho_T=-\frac{1}{2}.
  \end{aligned}
  \right.
\end{equation}
By simple calculation, we know that $(\gamma e^{\int_t^T(2r_s-\frac{B_s^2}{\sigma_s^2})\mathrm{d}s}, 0, -\frac{1}{2}e^{\int_t^Tr_s\mathrm{d}s})$ is a unique solution of ODEs \eqref{92362}. Then the limiting process $x^0$ is given by
\begin{equation*}
  \begin{aligned}
    &\mathrm{d}x_t^0=\Big\{r_tx_t^0-\frac{B^2_t\rho_t}{\sigma_t^2\Pi_t}-b_t\Big\}\mathrm{d}t-\frac{B_t\rho_t}{\sigma_t\Pi_t}\mathrm{d}W_t^0,\quad x_0^0=x_0.
  \end{aligned}
\end{equation*}

Therefore, we obtain the feedback representation of decentralized strategies as follows,
\begin{proposition}
    The optimal decentralized strategy of Problem (EX) has the following feedback representation,
    \begin{equation}\label{strategy}
        \bar u_t^i=-\frac{B_t(\hat{\bar x}_t^i-x_t^0)}{\sigma_t^2}-\frac{B_t\rho_t}{\sigma_t^2\Pi_t},
    \end{equation}
    where $(\Pi,\rho)$ is given by \eqref{92362} and $\hat{\bar x}^i$ satisfies
    \begin{equation*}
        \left\{
        \begin{aligned}
            &\mathrm{d}\bar x_t^i=\Big\{r_t\bar x_t^i-\frac{B^2}{\sigma^2}(\hat{\bar x}_t^i-x_t^0)-\frac{B^2_t\rho_t}{\sigma_t^2\Pi_t}-b_t\Big\}\mathrm{d}t-\Big\{\frac{B_t}{\sigma_t}(\hat{\bar x}_t^i-x_t^0)+\frac{B_t\rho_t}{\sigma_t\Pi_t}\Big\}\mathrm{d}W_t^0+c_t\mathrm{d}W_t^i+\bar c_t\mathrm{d}\overline W_t^i,\\
            &\bar x_0^i=x_0.
        \end{aligned}
        \right.
    \end{equation*}
\end{proposition}

Finally, a numerical example is given to illustrate the effectiveness of the proposed decentralized strategy. We assume that this LP system has $5000$ agents, let $T=1, x=2, r=0.06, \mu=0.15, b=0.06,  \sigma=0.25, c=0.5, \tilde\sigma=1, \check b=0.05, \check{\sigma}=1, \gamma=0.6, I=\tilde b=\check b=0$.  We show the trajectories of $(\bar x^{(N)},\bar u^{(N)})$ and $(x^0,u^0)$ in Fig~\ref{3} and Fig~\ref{4}. It can be seen that $(\bar x^{(N)},\bar u^{(N)})$ and $(x^0,u^0)$ coincide well, which illustrates the consistency of mean-field approximations.

\begin{figure}[htbp]
  \centering
  \begin{minipage}[t]{0.48\textwidth}
    \centering
    \includegraphics[width=\textwidth]{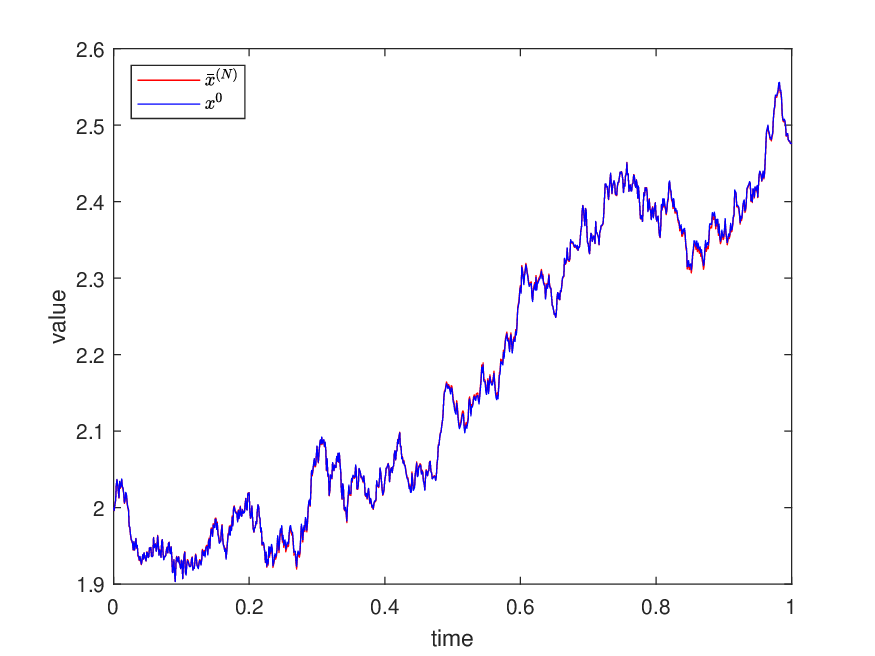}
    \caption{The trajectories of $\bar x^{(N)}$ and $x^0$.}
    \label{3}
  \end{minipage}%
  \hfill
  \begin{minipage}[t]{0.48\textwidth}
    \centering
    \includegraphics[width=\textwidth]{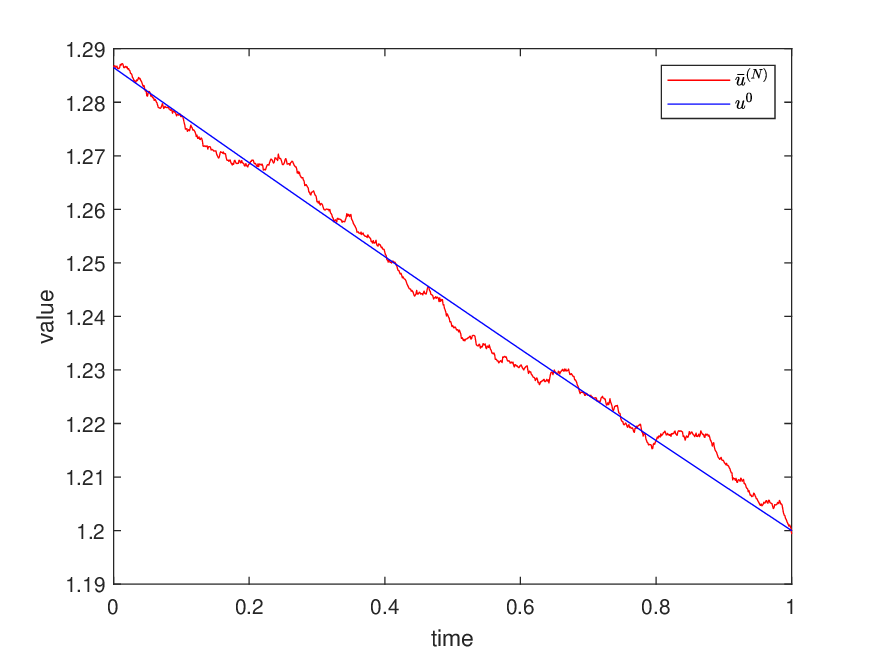}
    \caption{The trajectories of $\bar u^{(N)}$ and $u^0$.}
    \label{4}
  \end{minipage}
\end{figure}

\section{Conclusion}\label{sec:Conclusion}

In this paper, we investigate an indefinite MFG problem for the LP system with common noise, where both the state-average $x^{(N)}$ and control-average $u^{(N)}$ are involved. In our model, we allow the control variable to enter the diffusion term of the state, and the dynamic of each agent's state cannot be directly observed, but can be observed by {an individual} observation and a public observation. Furthermore, the weight matrices in the cost functional are indefinite, which can be zero or even negative. By applying the backward separation approach, we overcome the cyclic dependence between the control strategy and observation, and then obtain the optimal decentralized strategies using the Hamiltonian approach through FBSDE. To ensure the well-posedness of FBSDE \eqref{Hamiltonian}, which does not satisfy the monotonicity condition, we introduce the method of relaxed compensator. Then we provide the related CC system. In virtue of Riccati equation, we get the feedback representation of the optimal decentralized strategies and the well-posedness of Riccati equation by relaxed compensator. We creatively introduce the idea of equivalent cost functional into the original problem in LP system, then a class of equivalent cost functional have been obtained. Furthermore, we get an important result (see \eqref{important}) which is a key estimate to proving the $\varepsilon$-Nash equilibrium. Moreover, we solve a mean-variance portfolio selection problem to demonstrate the significance of our results.

\appendix
\section{Proof of Lemma \ref{lemma332}.}\label{Appendix A}

  From Definition \ref{def2131}, we have $\mathscr{U}_{d}^{i}\subseteq\mathscr{U}_{d}^{i,0}$, thus
  \begin{equation*}
    \inf_{\tilde{u}^i\in\mathscr{U}_{d}^{i}}J_i(\tilde{u}^i_{\cdot})\geq \inf_{u^i\in \mathscr{U}_{d}^{i,0}}J_i(u^i_{\cdot}).
  \end{equation*}
  Next, we prove the converse inequality.\\
  \noindent\textbf{Step 1}: $\mathscr{U}_{d}^{i}$ is dense in $\mathscr{U}_{d}^{i,0}$ with the norm of $L^2_{\mathscr{F}^{Y^{i,0}}}([0,T]; \mathbb{R}^m)$.

  For any $u^i\in \mathscr{U}_{d}^{i,0}$, define
  \begin{equation*}
     u_{n,t}^{i}=\left\{
      \begin{aligned}
        u_0,\qquad\qquad\qquad& \quad\text{ for } 0\leq t\leq \delta_n,\\
        \frac{1}{\delta_n}\int_{(k-1)\delta_n}^{k\delta_n}u_s^i\mathrm{d}s, &\quad \text{ for } k\delta_n<t\leq (k+1)\delta_n,
      \end{aligned}
    \right.
  \end{equation*}
  where $u_0\in \mathbb{R}^m$, $k, n$ are natural numbers, $1\leq k\leq n-1$, and $\delta_n=\frac{T}{n}$. Then $u_{n,t}^{i}$ is $\mathscr{F}_{k\delta_n}^{Y^{i,0}}$-adapted for any $k\delta_n<t\leq (k+1)\delta_n$, (we can use the classical progressively measurable modification of $u^i_{n}$ if necessary),
  and for any $n$
  \begin{equation}\label{star}
    \sup_{0\leq t\leq T} |u_{n,t}^{i}|\leq |u_0|+\sup_{0\leq t\leq T}|u_t^i|.
  \end{equation}
  Thus, $u_{n}^{i}\in\mathscr{U}_{d}^{i,0}$. Let $X_{n}^{i}$ and $Y^{i}_{n}$ be trajectories of \eqref{lxx} and \eqref{lYY} corresponding to $u^i_{n}$. For any $0\leq t\leq \delta_n$, we know that $u_{n,t}^{i}=u_0\in \mathcal U_{d}^{i}$. By virtue of Lemma \ref{lemma32}, we have $\mathscr{F}_{t}^{Y^{i,0}}=\mathscr{F}_{t}^{Y_n^{i}}$, $0\leq t\leq \delta_n$.  Next, for any $\delta_n < t\leq 2\delta_n$, we know that $u_{n,t}^{i}$ is $ \mathscr{F}_{\delta_n}^{Y^{i,0}}$ measurable. Since $ \mathscr{F}_{\delta_n}^{Y^{i,0}}=\mathscr{F}_{\delta_n}^{Y_n^{i}}$, then $u_{n,t}^{i}$ is $\mathscr{F}_{\delta_n}^{Y_n^{i}}$ measurable, thus we also know $u^i_{n,t}\mathbb I_{(\delta_n,2\delta_n]}(t)\in \mathcal U_{d}^{i}$ and $\mathscr{F}_{t}^{Y^{i,0}}=\mathscr{F}_{t}^{Y_n^{i}}$, $\delta_n < t\leq 2\delta_n$.   Similarly, step by step, for any $k\delta_n < t\leq (k+1)\delta_n$, we have $u_{n,t}^{i}$ is $ \mathscr{F}_{k\delta_n}^{Y^{i,0}}= \mathscr{F}_{k\delta_n}^{Y_n^{i}}$ measurable, and $\mathscr{F}_{t}^{Y^{i,0}}=\mathscr{F}_{t}^{Y_n^{i}}$, $k\delta_n < t\leq (k+1)\delta_n$.   Therefore, we have $\mathscr{F}_{t}^{Y^{i,0}}=\mathscr{F}_{t}^{Y_n^{i}}$ and $u^{i}_{n,t}$ is adapted to $\mathscr{F}_{t}^{Y^{i,0}}$ and $\mathscr{F}_{t}^{Y_n^{i}}$, which means that $u_{n}^{i}\in \mathscr{U}_{d}^{i}$.   Moreover, for each fixed $\omega$, $u_{n,t}^{i}\rightarrow u_t^i$ for almost every $t\in [0,T]$, when $n\rightarrow +\infty$.  Using \eqref{star}, by dominate convergence theorem, we derive $u_{n}^{i}\rightarrow u^i$ as $n\rightarrow +\infty$ in $L^2_{\mathscr{F}^{Y^{i,0}}}([0,T];\mathbb{R}^m)$, i.e., $\mathscr{U}_{d}^{i}$ is dense in $\mathscr{U}_{d}^{i,0}$.

  \noindent\textbf{Step 2}: $\lim_{n\rightarrow +\infty} J_i(u_{n,\cdot}^{i})=J_i(u_{\cdot}^i)$, where $u_{\cdot}^i, u_{n,\cdot}^{i}$ and $X^{i}_{n,\cdot}$ are defined as in Step 1. Noticing $\langle Qx,x\rangle-\langle Qy,y\rangle=\langle Q(x-y),x-y\rangle+2\langle Q(x-y),y\rangle$, then we have
  \begin{equation*}
    \begin{aligned}
      &2|J_i(u_{n,\cdot}^{i})\!-\!J_i(u^i_{\cdot})|\!\leq\!  C\Big\{\mathbb E\int_0^T\big|X_{n,t}^i\!-\!X_t^i\big|^2\mathrm{d}t\!+\!\mathbb E \int_0^T\big|u_{n,t}^i\!-\!u_t^i\big|^2\mathrm{d}t\\
      &\ \!+\!\Big[\Big(\mathbb E\int_0^T\big|X_t^i\!-\!\alpha_1x_t^0\!+\!1\big|^2\mathrm{d}t\Big)^{\frac{1}{2}}\!+\!\Big(\mathbb{E}\int_0^T|u_{n,t}^i\!-\!\beta_2 u_t^0|^2\mathrm{d}t\Big)^{\frac{1}{2}}\Big]\Big[\mathbb{E} \int_0^T \big|X_{n,t}^{i}\!-\!X_t^i\big|^2\mathrm{d}t \Big]^{\frac{1}{2}}\\
      &\  \!+\!\Big[\Big(\mathbb{E} \int_0^T|u_t^i\!-\!\beta_1 u_t^0\!+\!1|^2\mathrm{d}t \Big)^{\frac{1} {2}}\!+\!\Big(\mathbb{E}\int_0^T|X^{i}_{t}\!-\!\alpha_2x_t^0|^2\mathrm{d}t\Big)^{\frac{1}{2}}\Big] \Big[\mathbb{E} \int_0^T |u_{n,t}^{i}\!-\!u_t^i|^2\mathrm{d}t \Big]^{\frac{1}{2}} \\
      &\ \!+\!\mathbb E\big|X_{n,T}^i\!-\!X_T^i\big|^2\!+\!\Big[\mathbb{E}\big|X_T^i\!-\!\alpha_3x_T^0 \!+\!1\big|^2\Big]^{\frac{1}{2}} \Big[\mathbb{E}|X_{n,T}^{i}\!-\!X_T^i|^2\Big]^{\frac{1}{2}}\Big\}.
    \end{aligned}
  \end{equation*}
  By the standard estimation of SDE, we obtain that $J_i(u_{n,\cdot}^{i})\rightarrow J_i(u_{\cdot}^i)$ as $n \rightarrow +\infty$.

  \noindent\textbf{Step 3}: $\inf_{\tilde u^i\in\mathscr{U}_{d}^{i}}J_i(\tilde u_{\cdot}^i)\leq \inf_{u^i\in \mathscr{U}_{d}^{i,0}}J_i(u_{\cdot}^i)$.

  Since $u_{n,\cdot}^{i}\in\mathscr{U}_{d}^{i}$, we have
  $
    \inf_{\tilde u^i\in\mathscr{U}_{d}^{i}}J(\tilde u_{\cdot}^i)\leq J(u_{n,\cdot}^{i})$. By sending $n\rightarrow \infty$ and noticing Step 2, we have $\inf_{\tilde u^i\in\mathscr{U}_{d}^{i}}J_i(\tilde u_{\cdot}^i)\leq J_i(u_{\cdot}^i)$. Due to the arbitrariness of $u^i$, the desired inequality holds.

\section{Proof of Lemma \ref{Lem35}. }\label{Appendix B}

  For any $P \in \Upsilon([0,T];\mathbb S^n)$, applying It\^o's formula to $\langle P_tX_t^i, X_t^i\rangle$, we have
  \begin{equation*}\label{111621}\begin{aligned}
    -\langle P_0x,x\rangle &\!=\!\mathbb E\bigg[\!\int_0^T\!\big\{\langle (\dot{P}_t+P_tA_t+A_t^{\top}P_t)X_t^i+2[P_tb_t+D^{\top}_tP_t\bar b_t+(P_t\bar A_t+D^{\top}_tP_tD_t)x_t^0\\
    &\quad +(P_t\bar B_t+D^{\top}_tP_t\bar F_t)u_t^0], X_t^i\rangle +2\langle (P_tB_t+D^{\top}_tP_tF_t)u_t^i, X_t^i\rangle+\langle F^{\top}_tP_tF_tu_t^i\\
    &\quad +2(F^{\top}_tP_t\bar b_t+F^{\top}_tP_t\bar D_tx_t^0+F^{\top}_tP_t\bar F_tu_t^0), u_t^i\rangle +\langle \bar D_t^{\top}P_t\bar D_tx_t^0+2\bar D_t^{\top}P_t\bar b_t, x_t^0\rangle\\
    &\quad +\langle \bar F_t^{\top}P_t\bar F_tu_t^0+2\bar F_t^{\top}P_t\bar b_t, u_t^0\rangle +\bar b^{\top}_tP_t\bar b_t+\sigma_t^{\top}P_t\sigma_t+\bar\sigma_t^{\top}P_t\bar\sigma_t  \big\}\mathrm{d}t-\langle P_TX_T^i, X_T^i\rangle \bigg]\\
    &\!=\!\mathbb E\bigg[\!\int_0^T\!\big\{\langle(Q^P_t-Q_t)X_t^i+2(q^P_t-q_t+\alpha_1Q_tx_t^0+\beta_2S_tu_t^0),X_t^i\rangle+\langle (R_t^P-R_t)u_t^i    \\
    &\quad +2(r^P_t-r_t+\alpha_2S^{\top}_tx^0_t+\beta_1 R_tu_t^0),u_t^i\rangle+2\langle (S^P_t-S_t) u_t^i, X_t^i\rangle-\alpha_1\langle \alpha_1 Q_tx_t^0-2q_t,x_t^0\rangle\\
    &\quad -\beta_1\langle \beta_1R_tu_t^0-2r_t,u_t^0\rangle-2\alpha_2\beta_2\langle S_tu_t^0,x_t^0\rangle+M_t^P\big\}\mathrm{d}t+\langle (L_T^P-L_T)X_T^i,X_T^i\rangle \bigg].
  \end{aligned}\end{equation*}
  By adding the above equation to both sides of the cost \eqref{lcost}, then the desired result is obtained.

  Moreover, let $P$ be a relaxed compensator for Problem (MFD). According to Definition \ref{def36}, $(Q^P,S^P,R^P,L^P_T)$ satisfies Condition (PD). Thus for any given initial value $x\in \mathbb R^n$ and any $u^i\in \mathscr{U}_{d}^{i}$, by virtue of Remark \ref{Rem31}, we know that Problem (MFP) is well-posed. Then Problem (MFD) is also well-posed by Lemma \ref{Lem35}.

  \section{Proof of Theorem \ref{thm8314}. }\label{Appendix D} Firstly, we prove the unique solvability of Riccati equation \eqref{Riccati}.
  If there exists a relaxed compensator $P\in\Upsilon([0,T];\mathbb S^n)$, we can know that the quadruple $(Q^P,S^P,R^P,L^P_T)$ satisfies Condition (PD) by Definition \ref{def36}. Thus by virtue of \cite[Theorem 7.2]{Yong1999}, the following Riccati equation
  \begin{equation}\label{pRiccati}
  \left\{
  \begin{aligned}
     &\dot{\Pi}_t^P+\Pi_t^P A_t+A_t^{\top}\Pi_t^P+D_t^{\top}\Pi_t^P D_t+Q_t^P-[(S_t^P)^{\top}+B_t^{\top}\Pi_t^P +F_t^{\top}\Pi_t^P D_t]^{\top}\\
     &\qquad \times(R_t^P+F_t^{\top}\Pi_t^P F_t)^{-1} [(S_t^P)^{\top}+B_t^{\top}\Pi_t^P+F_t^{\top}\Pi_t^P D_t]= 0,\\
     &\Pi^P_T=L^P_T,\quad R^P+F^{\top}\Pi^P F\gg 0,\quad t\in [0,T],
  \end{aligned}
  \right.
\end{equation}
admits a unique solution $\Pi^P\in C([0,T];\mathbb S^n_+)$. Then one can check easily
\begin{equation}\label{trans}
  \Pi=\Pi^P+P,
\end{equation}
solves Riccati equation \eqref{Riccati}. Moreover, if $\Pi\in C([0,T];\mathbb S^n)$ solves Riccati equation \eqref{Riccati}, then the inverse transformation \eqref{trans} provides a solution to \eqref{pRiccati}. Thus the solvability of Riccati equation \eqref{Riccati} and \eqref{pRiccati} are equivalent.

Secondly, we derive the feedback representation of the decentralized strategy $\bar u^i$.  From \eqref{pHamiltonian} and \eqref{px0}, we have
  \begin{equation}\label{x0}
      \begin{aligned}
        &\mathrm{d}x_t^0\!=\!\big[(A_t\!+\!\bar A_t)x_t^0\!+\!(B_t\!+\!\bar B_t)u^0_t\!+\!b_t\big] \mathrm{d}t \!+\!\big[(D_t\!+\!\bar D_t)x_t^0\!+\!(F_t\!+\!\bar F_t)u^0_t\!+\!\bar b_t\big]\mathrm{d}W_t^0,\quad x_0^0\!=\!x.
      \end{aligned}
  \end{equation}
  Noticing the terminal condition of \eqref{Hamiltonian}, we can suppose
  \begin{equation}\label{112035}
    {\varphi}_t^i=\Pi_t ({\bar X}_t^i-x_t^0)+\Sigma_tx_t^0+\rho_t,
  \end{equation}
  where $\Pi:[0,T]\rightarrow \mathbb S^n$, $\Sigma:[0,T]\rightarrow \mathbb S^n$ and $\rho :[0,T]\rightarrow \mathbb R^n$ are absolutely continuous functions with terminal condition $\Pi_T=L_T$, $\Sigma_T=(1-\alpha_3)L_T$ and $\rho_T=l_T$, respectively. Applying It\^o's formula to \eqref{112035}, we have
   \begin{equation}\label{062440}
     \begin{aligned}
       &\mathrm{d}\varphi_t^i=\big\{\dot{\Pi}_t(\bar X_t^i-x_t^0)+\Pi_t[A_t(\bar X_t^i-x_t^0)+B_t(\bar u_t^i-u_t^0)]+\dot{\Sigma}_tx_t^0+\Sigma_t\big[(A_t+\bar A_t)x_t^0\\
       &\qquad\quad +(B_t+\bar B_t)u_t^0+ b_t\big]+\dot{\rho}_t\big\}\mathrm{d}t+\big\{\Pi_t[D_t(\bar X_t^i-x_t^0)+F_t(\bar u_t^i-u_t^0)]\\
       &\qquad\quad +\Sigma_t[ (D_t+\bar D_t)x_t^0+ (F_t+\bar F_t)u_t^0+\bar b_t]\big\}\mathrm{d}W_t^0+\Pi_t\sigma_t\mathrm{d}W_t^i+\Pi_t\bar\sigma_t\mathrm{d}\overline W_t^i.
     \end{aligned}
   \end{equation}
   Comparing \eqref{062440} with \eqref{Hamiltonian}, it yields,
  \begin{equation}\label{eta412}
      {\eta}_t^i= \Pi_t[D_t(\bar X_t^i-x_t^0)+F_t(\bar u_t^i-u_t^0)]+\Sigma_t[ (D_t+\bar D_t)x_t^0+ (F_t+\bar F_t)u_t^0+\bar b_t].
  \end{equation}
  Taking $\mathbb E[\cdot|\mathscr F^{\bar Y^i}]$ both on the drift terms of \eqref{062440} and  \eqref{Hamiltonian}, then we obtain
  \begin{equation*}
    \begin{aligned}
      0&=\big(\dot{\Pi}_t+\Pi_t A_t+A_t^{\top}\Pi_t+D_t^{\top}\Pi_t D_t+Q_t- \widetilde \Pi_t\mathcal{R}_{t}^{-1}\widetilde \Pi_t^{\top}\big)(\hat{\bar X}_t^i-x_t^0)\\
      &\quad+ [\dot{\Sigma}_t+\Sigma_t(A_t+\bar A_t)+A_t^{\top}\Sigma_t+D^{\top}_t\Sigma_t(D_t+\bar D_t) -\widetilde\Sigma_t \widetilde{\mathcal R}^{-1}_t\bar\Sigma_t^{\top}\\
      &\quad+(1-\alpha_1)Q_t ]x_t^0+\dot{\rho}_t+A_t^{\top}\rho_t-\widetilde\Sigma_t \widetilde{\mathcal R}_t^{-1}\tilde \rho_t+\Sigma_tb_t+ D^{\top}_t\Sigma_t\bar b_t+q_t.
    \end{aligned}
  \end{equation*}
  Then we can obtain the Riccati equations \eqref{Riccati}-\eqref{Riccati2} and ODE \eqref{BSDE0}. By substituting \eqref{112035} and \eqref{eta412} into \eqref{optimalcontrol}, we obtain
  \begin{equation*}
      \begin{aligned}
          &B^{\top}_t[\Pi_t(\hat{\bar X}_t^i-x_t^0)+\Sigma_tx_t^0+\rho_t]+F_t^{\top}\Pi_t[D_t(\hat{\bar X}_t^i-x_t^0)+F_t(\bar u_t^i-u_t^0)]\\
          &\quad +F^{\top}_t\Sigma_t[(D_t+\bar D_t)x_t^0+(F_t+\bar F_t)u_t^0+\bar b_t]+S_t(\hat{\bar X}_t^i-\alpha_2x_t^0)+R_t(\bar u_t^i-\beta_1u_t^0)+r_t=0.
      \end{aligned}
  \end{equation*}
  Then, recalling $u^0=\mathbb E[\bar u^i|\mathscr F^{\theta}]$ (see \eqref{px0}), we further derive
  \begin{equation}\label{u0412}
      u_{t}^0=-\widetilde{\mathcal R}_t^{-1}(\bar\Sigma_t^{\top}x_t^0+\tilde\rho_t),
  \end{equation}
  along with the feedback representation given in \eqref{feed230129}.

  Finally, we prove $\bar u_i$ given by \eqref{feed230129} belongs to $\mathcal U_d^i$. Substituting \eqref{u0412} into \eqref{x0}, we obtain  \eqref{x0412}. Recalling (H1)-(H2),  \eqref{x0412} is a linear SDE with uniformly bounded coefficients, which implies the unique solvability of \eqref{x0412}. Based on the classical estimate of solution to SDE, we further obtain $\mathbb E[\sup_{0\leq t\leq T}|x_t^0|^2]\leq C$, where $C$ is a positive constant depending on $x, \Pi, \Sigma, \rho$ and the uniformly bound of all coefficients. Thus, we also have $\mathbb E[\sup_{0\leq t\leq T}|u_t^0|^2]\leq C$ by \eqref{u0412}. Then, from \eqref{feedback x}, it holds that $\mathbb E[\sup_{0\leq t\leq T}|\bar X_t^i|^2]\leq C$. Hence, we have $\mathbb E[\sup_{0\leq t\leq T}|\bar u_t^i|^2]\leq C$ from \eqref{feed230129}. In addition, recalling the notation \eqref{filteringsymbol}, and noting that $\hat{\bar X}^i=\mathbb E[\bar X^i|\mathscr F^{\bar Y^i}]$ is $\mathscr F^{\bar Y^i}$ adapted and $x^0$ is $\mathscr F^{\theta}$ adapted,  it follows that $\bar u^i$ is $\mathscr F^{\bar Y^i}$ adapted. Therefore, one can obtain $\bar u^i\in \mathcal U_d^i$.

  \section{Proof of Lemma \ref{rmk031337}.} \label{Appendix C}
  (i) $\Rightarrow$ (ii): By the first conclusion of  Theorem \ref{thm8314}, one can obtain the desired result.

  (ii) $\Rightarrow$ (i): When Riccati equation \eqref{Riccati} admits a unique solution, one can check that $\Pi$ satisfies Condition (RC). Then by Proposition \ref{prop310}, we know that $\Pi$ is a relaxed compensator.

  (ii) $\Leftrightarrow$ (iii): By virtue of \cite[Theorem 4.5]{SLY2016}, we can obtain the equivalence between (ii) and (iii).

   Next, we provide a further explanation of (i) $\Rightarrow$ (iii).

  Noticing Definition \ref{def36}, we obtain that the map $u^i\rightarrow J_i^P(u_{\cdot}^i)$ is uniformly convex by
  \cite[Corollary 3.4 and Proposition 3.5]{SLY2016}. Thus (see  \cite[Equation (3.7)]{SLY2016}), we obtain for any $u^i\in \mathscr U_{d}^{i}$ and some $\lambda>0$,
  \begin{equation*}\begin{aligned}
    J_i^{P,0}(u_{\cdot}^i)&\!:=\!\frac{1}{2}\mathbb{E}\bigg[\!\int_0^T(\langle Q^P_t\mathbb X_t^i ,\mathbb X_t^i\rangle\! +\!\langle R^P_tu_t^i, u_t^i\rangle
    \!+\!2\langle S^P_tu_t^i, \mathbb X_t^i\rangle )\mathrm{d}t\!+\!\langle L^P_T\mathbb X_T^i, \mathbb X_T^i\rangle\bigg]\\
    &\! \geq\! \lambda \mathbb E\int_0^T\!\!|u_t^i|^2\mathrm{d}t,
  \end{aligned}\end{equation*}
  where $\mathbb X^i$ is defined by the following SDE,
  \begin{equation*}
    \begin{aligned}
      \mathrm{d}\mathbb X_t^i=\big\{A_t\mathbb X_t^i+B_tu_t^i\big\} \mathrm{d}t  +\big\{D_t\mathbb X_t^i+F_tu_t^i\big\}\mathrm{d}W_t^0,\quad
      \mathbb X_0^i=0.
    \end{aligned}
  \end{equation*}
  Similar to Lemma \ref{Lem35}, we know that when 
  $x=0$, it follows that $J^0_i(u_{\cdot}^i)=J^{P,0}_i(u_{\cdot}^i)$,  where
  \begin{equation*}
    J^0_i(u_{\cdot}^i):=\frac{1}{2}\mathbb{E}\bigg[\int_0^T\big\{\langle Q_t\mathbb X_t^i, \mathbb X_t^i\rangle +\langle R_tu_t^i, u_t^i\rangle +2\langle S_tu_t^i, \mathbb X_t^i\rangle\big\}\mathrm{d}t+\langle L_T\mathbb X_T^i,\mathbb X_T^i\rangle\bigg].
  \end{equation*}
  Then we have $J^0_i(u_{\cdot}^i)\geq \lambda \mathbb E\int_0^T|u_t^i|^2\mathrm{d}t$ for any $u^i\in \mathscr U_{d}^{i}$ and some $\lambda>0$, which yields that the map $u^i\rightarrow J_i(u_{\cdot}^i)$ is uniformly convex by \cite[Proposition 3.5]{SLY2016}.



\begin{thebibliography}{99}

\bibitem{A1969} Albert A (1969) Conditions for positive and nonnegative definiteness in terms of pseudoinverses.  \emph{SIAM J. Appl. Math.}  17(2): 434-440.

\bibitem{AM2007} Anderson BDO, Moore JB (1989) \emph{Optimal Control: Linear Quadratic Methods} (Prentice-Hall, New Jersey).

\bibitem{BFH2021} Bensoussan A, Feng X, Huang J (2021) Linear-quadratic-Gaussian mean-field-game with partial observation and common noise. \emph{Math. Control Relat. Fields} 11(1): 23-46.

\bibitem{BFY2013} Bensoussan A, Frehse J, Yam SCP (2013) \emph{Mean Field Games and Mean Field Type Control Theory} (Springer, New York).

\bibitem{Carmona2018} Carmona R, Delarue F (2018) \emph{Probabilistic Theory of Mean Field Games with Applications I-II} (Springer Nature, Berlin).

\bibitem{CDL2016} Carmona R, Delarue F, Lacker D (2016) Mean field games with common noise. \emph{Ann. Probab.} 44(6): 3740-3803.

\bibitem{CLZ1998} Chen S, Li X, Zhou XY (1998) Stochastic linear quadratic regulators with indefinite control weight costs. \emph{SIAM J. Control Optim.} 36(5): 1685-1702.

\bibitem{CZ2000} Chen S, Zhou XY (2000) Stochastic linear quadratic regulators with indefinite control weight costs II. \emph{SIAM J. Control Optim.} 39(4): 1065-1081.

\bibitem{CDW2024} Chen T, Du K, Wu Z (2024) Partially observed mean-field game and related mean-field forward-backward stochastic differential equation. \emph{J. Differential Equations} 408: 409-448.

\bibitem{ET2015} Espinosa GE, Touzi N (2015) Optimal investment under relative performance concerns. \emph{Math. Financ.} 25(2): 221-257.

\bibitem{GXZ2023} Guan C, Xu ZQ, Zhou R (2023) Dynamic optimal reinsurance and dividend payout in finite time horizon. \emph{Math. Oper. Res.} 48(1): 544-568.

\bibitem{Hu20182} Hu Y, Huang J, Nie T (2018) Linear-quadratic-Gaussian mixed mean-field games with heterogeneous input constraints. \emph{SIAM J. Control Optim.} 56(4): 2835-2877.

\bibitem{HSX2023} Hu Y, Shi X, Xu ZQ (2023) Constrained monotone mean-variance problem with random coefficients. \emph{SIAM J. Financ. Math.} 14(3): 838-854.

\bibitem{HMC2007} Huang M, Caines PE, Malham\'e RP (2007) Large-population cost-coupled LQG problems with nonuniform agents: Individual-mass behavior and decentralized $\varepsilon $-Nash equilibria. \emph{IEEE Trans. Autom. Control} 52(9): 1560-1571.

\bibitem{K1960} Kalman RE (1960) Contributions to the theory of optimal control. \emph{Bol. Soc. Mat. Mexicana} 5(2): 102-119.

\bibitem{KT2019} Kolokoltsov VN, Troeva M (2019) On mean field games with common noise and McKean-Vlasov SPDEs. \emph{Stoch. Anal. Appl.} 37(4): 522-549.

\bibitem{LL2007} Lasry JM, Lions PL (2007) Mean field games. \emph{Jpn. J. Math.} 2(1): 229-260.

\bibitem{LZ1999} Lim AEB, Zhou XY (1999) Stochastic optimal LQR control with integral quadratic constraints and indefinite control weights. \emph{IEEE Trans. Autom. Control} 44(7): 1359-1369.

\bibitem{Liptser1977} Liptser RS, Shiryaev AN (2013) \emph{Statistics of Random Processes: I-II} (Springer Science \& Business Media).

\bibitem{LLY2020} Li N, Li X, Yu Z (2020) Indefinite mean-field type linear-quadratic stochastic optimal control problems. \emph{Automatica} 122: 109267.

\bibitem{LT1995} Li X, Tang S (1995) General necessary conditions for partially observed optimal stochastic controls. \emph{J. Appl. Probab.} 32(4): 1118-1137.

\bibitem{LZ2002} Lim AEB, Zhou XY (2002) Mean-variance portfolio selection with random parameters in a complete market. \emph{Math. Oper. Res.} 27(1): 101-120.

\bibitem{Majerek2005} Majerek D, Nowak W, Zieba W (2005) Conditional strong law of large number. \emph{Int. J. Pure Appl. Math} 20(2): 143-156.

\bibitem{M1952} Markowitz H (1952) Portfolio selection. \emph{J. Financ.} 7: 77-91.

\bibitem{M1960} Markowitz H (1960) \emph{Portfolio Selection: Efficient Diversification of Investment} (John Wiley and Sons, New York).

\bibitem{RMZ2002} Rami MA, Moore JB, Zhou XY (2002) Indefinite stochastic linear quadratic control and generalized differential Riccati equation. \emph{SIAM J. Control Optim.} 40(4): 1296-1311.


\bibitem{SLY2016} Sun J, Li X, Yong J (2016) Open-loop and closed-loop solvabilities for stochastic linear quadratic optimal control problems. \emph{SIAM J. Control Optim.} 54(5): 2274-2308.

\bibitem{SXY2021} Sun J, Xiong J, Yong J (2021) Indefinite stochastic linear-quadratic optimal control problems with random coefficients: Closed-loop representation of open-loop optimal controls. \emph{Ann. Appl. Probab.} 31(1): 460-499.

\bibitem{SY2020} Sun J, Yong J (2020) \emph{Stochastic Linear-Quadratic Optimal Control Theory: Open-Loop and Closed-Loop Solutions} (Springer Nature, Switzerland).

\bibitem{T1998} Tang S (1998) The maximum principle for partially observed optimal control of stochastic differential equations. \emph{SIAM J. Control Optim.} 36(5): 1596-1617.

\bibitem{Wang2020} Wang BC, Zhang H, Zhang JF (2020) Mean field linear-quadratic control: Uniform stabilization and social optimality. \emph{Automatica} 121: 109088.

\bibitem{Wang20202} Wang BC, Zhang H (2020) Indefinite linear quadratic mean field social control problems with multiplicative noise. \emph{IEEE Trans. Autom. Control} 66(11): 5221-5236.

\bibitem{WW2008} Wang G, Wu Z (2008) Kalman-Bucy filtering equations of forward and backward stochastic systems and applications to recursive optimal control problems. \emph{J. Math. Anal. Appl.} 342(2): 1280-1296.

\bibitem{Wang2015} Wang G, Wu Z, Xiong J (2015) A linear-quadratic optimal control problem of forward-backward stochastic differential equations with partial information. \emph{IEEE Trans. Autom. Control} 60(11): 2904-2916.

\bibitem{WWX2018} Wang G, Wu Z, Xiong J (2018) \emph{An introduction to optimal control of FBSDE with incomplete information} (Springer Cham, Berlin).

\bibitem{W1968} Wonham WM (1968) On the separation theorem of stochastic control. \emph{SIAM J. Control Optim.} 6(2): 312-326.

\bibitem{XXZ2021} Xiong J, Xu ZQ, Zheng J (2021) Mean-variance portfolio selection under partial information with drift uncertainty. \emph{Quant. Financ.} 21(9): 1461-1473.

\bibitem{Xu2020} Xu R, Zhang F (2020) $\varepsilon$-Nash mean-field games for general linear-quadratic systems with applications. \emph{Automatica} 114: 108835.

\bibitem{Yong1999} Yong J, Zhou XY (1999) \emph{Stochastic Controls: Hamiltonian Systems and HJB Equations} (Springer Science \& Business Media, New York).

\bibitem{Y1974} Yu PL (1974) Cone convexity, cone extreme points, and nondominated solutions in decision problems with multiobjectives. \emph{J. Optim. Theory Appl.} 14: 319-377.

\bibitem{Y2013} Yu Z (2013) Equivalent cost functionals and stochastic linear quadratic optimal control problems. \emph{ESAIM-Control Optim. Calc. Var.} 19(1): 78-90.

\bibitem{ZL2000} Zhou XY, Li D (2000) Continuous-time mean-variance portfolio selection: A stochastic LQ framework. \emph{Appl. Math. Optim.} 42: 19-33.

\bibitem{ZY2003} Zhou XY, Yin G (2003) Markowitz's mean-variance portfolio selection with regime switching: A continuous-time model. \emph{SIAM J. Control Optim.} 42(4): 1466-1482.

\end{thebibliography}
\end{document}